\documentclass{amsart}

\usepackage{amssymb,amsfonts}
\usepackage[all,arc]{xy}
\usepackage{enumerate, bbm}%https://www.sharelatex.com/project/584ab56a990240004f08c360
\usepackage{mathrsfs}
\usepackage{mathtools}
\usepackage{hyperref}
\usepackage[usenames, dvipsnames]{color}%colors
\usepackage{graphicx}
\usepackage{tikz}
\usepackage{tabularx,ragged2e,booktabs,caption}
\graphicspath{ {TexImages/} }
\usepackage{import}

\newtheorem{thm}{Theorem}[section]
\newtheorem{cor}[thm]{Corollary}
\newtheorem{prop}[thm]{Proposition}
\newtheorem{lem}[thm]{Lemma}
\theoremstyle{definition}

\newtheorem{quest}{Question}
\newtheorem{conj}[thm]{Conjecture}
\theoremstyle{remark}
\newtheorem{rem}[thm]{Remark}

\newcommand{\half}{\frac{1}{2}}

\newcommand{\0}{\mathbf{0}}

\newcommand{\TODO}{\textbf{\textcolor{red}{T}\textcolor{green}{O}\textcolor{blue}{D}\textcolor{BurntOrange}{O}}}

\newcommand{\set}[1]{\left\{#1\right\}}
\newcommand{\de}[1]{\textit{#1}}
\newcommand{\sub}{\subseteq}

\newcommand\bovermat[2]{%
  \makebox[0pt][l]{$\smash{\overbrace{\phantom{%
    \begin{matrix}#2\end{matrix}}}^{\text{#1}}}$}#2}
\newcommand{\dense}[1]{\widehat{#1}}

\DeclareMathOperator*{\forb}{forb}
\DeclareMathOperator*{\Avoid}{Avoid}

\hypersetup{
	colorlinks,
	citecolor=red,
	filecolor=red,
	linkcolor=red,
	urlcolor=red,
	linktocpage
}

\title{Forbidden Families of Minimal Quadratic and Cubic Configurations}
\author{Attila Sali}
\address{Alfr\'ed R\'enyi Institute of Mathematics\\Hungarian Academy of Sciences}
\thanks{Research was partially supported by Hungarian National Research, Development and Innovation Office – NKFIH, K116769}
\author{Sam Spiro}
\address{University of Miami}
\thanks{Research was done while the second author took the \emph{Research Opportunties} course at Budapest Semesters in Mathematics under the supervision of the first author}
\begin{document}

\maketitle
%\tableofcontents
\begin{abstract}
     A matrix is \emph{simple} if it is a (0,1)-matrix and there are no repeated columns. Given a (0,1)-matrix $F$, we say a matrix $A$ has $F$ as a \emph{configuration}, denoted $F\prec A$, if there is a submatrix of $A$ which is a row and column permutation of $F$. Let $|A|$ denote the number of columns of $A$. Let $\mathcal{F}$ be a family of matrices. We define the extremal function $\forb(m, \mathcal{F}) = \max\{|A|\colon A \text{ is an }m-\text{rowed simple matrix and has no configuration } F\in\mathcal{F}\}$. We consider pairs $\mathcal{F}=\{F_1,F_2\}$ such that $F_1$ and $F_2$ have no common extremal construction and derive that individually each $\forb(m, F_i)$ has greater asymptotic growth than $\forb(m, \mathcal{F})$, extending research started by Anstee and Koch \cite{AK}.
\end{abstract}
\section{Introduction}

The investigations into the extremal problem of the maximum number of edges in an $n$ vertex graph with no subgraph $H$ originated with Erd\H os and Stone \cite{ES} and Erd\H os and Simonovits \cite{Simon} . There is a large and illustrious literature. A natural  extension to general hypergraphs is to forbid a given \emph{trace}. This latter problem in the language of matrices is our focus. We say a matrix is \emph{simple} if it is a (0,1)-matrix and there are no repeated columns. Given a (0,1)-matrix $F$, we say a matrix $A$ has $F$ as a \emph{configuration}, denoted $F\prec A$, if there is a submatrix of $A$ which is a row and column permutation of $F$. Let $|A|$ denote the number of columns in $A$. We define
\[\Avoid(m,F)=\left\{A\,:\,A\hbox{ is }m\hbox{-rowed simple}, F\nprec A\right\},\]
\[\forb(m,{F})=\max_A\{|A|: A \in\Avoid(m,{F})\}.\]
A simple (0,1)-matrix $A$ can be considered as vertex-edge incidence matrix of a hypergraph without repeated edges. A configuration is a trace of a subhypergraph of this hypergraph. 

Let $A^c$ denote the 0-1-complement of a (0,1)-matrix $A$. It is easy to see that $\forb(m,{F})=\forb(m,{F}^c)$.

 We recall an important conjecture from \cite{AS05}. Let $I_k$ denote the $k\times k$ identity matrix, let $I^c_k$ denote the (0,1)-complement of $I_k$, and let $T_k$ denote the $k\times k$ upper triangular matrix whose $i$th column  has 1's in rows $1,2,\ldots ,i$ and 0's in the remaining rows. For $p$ matrices $m_1\times n_1$ matrix $A_1$, an $m_2\times n_2$ matrix $A_2$,$\dots$, an $m_p\times n_p$ matrix $A_p$  we define $A_1\times A_2\times\cdots\times A_p$ as the $(m_1+\cdots +m_p)\times n_1n_2\cdots n_p$ matrix whose columns consist of all possible combinations obtained from placing a column of $A_1$ on top of a column of $A_2$ on top of a column of $A_3$ etc. For example, the vertex-edge incidence matrix of the complete bipartite graph $K_{m/2,m/2}$ is $I_{m/2}\times I_{m/2}$. Define $1_k$ to be the $k\times 1$ column of 1's and $0_{\ell}$ to be the $\ell\times 1$ column of 0's. 

\begin{conj}\cite{AS05} Let $F$ be a $k\times \ell$ matrix with $F\ne \begin{bmatrix}0\\ 1\end{bmatrix}$. Let $X(F)$ denote the largest $p$ such that there are choices $A_1,A_2,\ldots ,A_p\in\{I_{m/p},I^c_{m/p},T_{m/p}\}$ so that $F\nprec A_1\times A_2\times\cdots \times A_p$. Then $\forb(m,F)=\Theta(m^{X(F)})$.\label{grand}\end{conj}

We are assuming $p$ divides $m$ which does not affect asymptotic bounds.

It is natural to extend the concepts of $\Avoid(m,F)$ and $\forb(m,F)$ to the case when not just a single configuration, but a family $\mathcal{F}=\{F_1,F_2,\ldots ,F_r\}$ of configurations is forbidden. 
\[\Avoid(m,\mathcal{F})=\left\{A\,:\,A\hbox{ is }m\hbox{-rowed simple}, F\nprec A\hbox{ for all } F\in\mathcal{F} \right\},\]
\[\forb(m,\mathcal{F})=\max_A\{|A|: A \in\Avoid(m,\mathcal{F})\}.\]
One important result in this area is the following theorem of Balogh and Bollob\'as \cite{BB}.
\begin{thm}[Balogh and Bollob\'as, 2005]\label{thm:BB}
For a given $k$, there is a constant $BB(k)$ such that $\forb(m,\{I_k,T_k,I^c_k\})=BB(k)$.
\end{thm}
The best current estimate for $BB(k)$ is due to Anstee and Lu \cite{ALU}, $BB(k)\le 2^{ck^2}$ where $c$ is absolute constant, independent of $k$.
It could be tempting to extend Conjecture~\ref{grand} to the case of forbidden families, as well. However, as it was shown in \cite{AKRS} $\forb(m,\left\{I_2\times I_2,T_2\times T_2\right\})$ is $\Theta(m^{3/2})$ despite the only products missing both $I_2\times I_2$ and $T_2\times T_2$.are  one-fold products. An even stronger observation is made in Remark~\ref{rem:no-conj}.

In the present paper we continue the investigations started in \cite{AK}. Anstee and Koch determined $\forb(m,\{F,G\})$ for all pairs $\{F,G\}$, where both members are \emph{minimal quadratics}, that is both $\forb(m,F)=\Theta(m^2)$ and $\forb(m,G)=\Theta(m^2)$, but no proper subconfiguration of $F$ or $G$ is quadratic. We take this one step further.  That is, we consider cases when one of $F$ or $G$ is a simple minimal cubic configuration and the other one is a minimal quadratic or minimal simple cubic. Our results are summarized in Table~\ref{tab::results}. We solve all cases when the minimal simple cubic configuration has four rows. If Conjecture~8.1 of  \cite{survey} is true, then there are no minimal simple cubic configurations on 5 rows. The six-rowed ones are discussed in Section~8. The remaining case is $\forb(m,Q_8,F_{14})$, where we believe that non-existence of common quadratic product construction indicates that the order of magnitude is $o(m^2)$.

The structure of the paper is as follows. In Section~2 product constructions and bounds implied by them are treated. Then in Section~3 upper bounds implied by the \textit{standard induction} technique (\cite{survey}, Section~11) are given. These combined with product constructions give asymptotically sharp bounds for many pairs of configurations. Sections~4, 5, 6 and 7 deal with specific configurations. In Section~4 a stability theorem is proven for matrices avoiding the configuration $Q_3(t)$, which is a generalization of the configuration $Q_3$ (see Table~\ref{tab:minquad}), and this theorem is applied to prove forbidden pairs results involving $Q_3(t)$. Section~5 contains cases when one member of the forbidden pairs is a block of 1's. This naturally involves extremal graph and hypergraph results, as forbidding $1_{k,1}$ restricts the hypergraph corresponding to our simple (0,1)-matrix to be of \textit{rank}-$(k-1)$, that is edges are of size at most $k-1$. Interestingly enough, in one case we use a very recent theorem  of Alon and Shikhelman \cite{ash} combined with an old fundamental result of F\"uredi \cite{furedi-kernel}. Section~6 considers $F_9$ (see Table~\ref{tab:mincub}). Interestingly, some exact results are also obtained.  Section~7 deals with $Q_9$ of Table~\ref{tab:minquad} based on the characterization  of $Q_9$ avoiding matrices of \cite{ABS11}. Finally, in Section~8 we observe that $\forb(m,\set{F,G})$ is quadratic if $F$ is a minimal quadratic and $G$ is a 6-rowed minimal cubic in all but one case. 

Throughout the paper we use standard extremal graph and hypergraph notations, such as $ex(m,G)$ to denote the largest number of edges a graph on $m$ vertices can have without containing a subgraph isomorphic to $G$, or $ex^{(k)}(m,\mathcal{H})$ for the largest number of edges a $k$-uniform hypergraph can have without containing a subhypergraph $\mathcal{H}$. The complete $k$-partite $k$-uniform hypergraph on partite sets of sizes $s_1,\ldots,s_k$, respectively is denoted by $K(s_1,\ldots,s_k)$. 
Also, when forbidden pairs of configurations are considered, we use the notational simplification $\forb(m,\set{F,G})=\forb(m,F,G)$ for typesetting convenience. We allow ourselves the ambiguity of writing $I\times I^c$ instead of the technically precise $I_{m/2}\times I^c_{m/2}$ in product constructions.

\section{Product Constructions}

What follows are tables of all minimal quadratic configurations and simple minimal cubic configurations with 4 rows.  In addition to the configurations, we have included a list of all 2-fold and 3-fold products of $I,\ I^c$ and $T$ that avoid these configurations.  The list of constructions avoiding quadratic configurations comes from \cite{AK}, and the lists for cubic configurations are proved in Section~2, with the statement that proves the result listed under ``Proposition.'' 

\begin{center}
\captionof{table}{Minimal Quadratic Configurations}\label{tab:minquad}
\begin{tabular}{ | c | c | c | } 
\hline
 & Configuration  $Q_i$ &  Construction(s) \\ 
\hline
$1_{3,1}$ & $\begin{bmatrix}1\\ 1\\  1\end{bmatrix}$ & $I\times I$ \\ 
\hline
$1_{2,2}$ & $\begin{bmatrix}1 & 1\\ 1 & 1\end{bmatrix}$ & $I\times I$ \\ 
\hline
$I_3$ & $\begin{bmatrix}1 & 0 & 0\\ 0 & 1 & 0\\ 0 & 0 & 1\end{bmatrix}$ & $\begin{matrix}I^c\times I^c \\ I^c\times T\\ T\times T \end{matrix}$\\ 
\hline
$Q_3$  & $\begin{bmatrix}0 & 0 & 0 & 1 & 1 & 1\\ 0 & 1 & 1 & 0 & 0 & 1\end{bmatrix}$ & $I\times I^c$\\ 
\hline
$Q_8$ & $\begin{bmatrix}0 & 0 & 1 & 1\\ 1 & 0 & 1 & 0\\ 0 & 1 & 0 & 1\end{bmatrix}$ & $T\times T$\\ \hline 
$Q_9$  & $\begin{bmatrix}1 & 0\\ 1 & 0\\ 0 & 1\\ 0 & 1\end{bmatrix}$ & $\begin{matrix} I\times T\\ I^c\times T\end{matrix}$\\ 
\hline

\end{tabular}
\end{center}
Note that we have not included the complements of $1_{3,1},\ 1_{2,2},\ $and $I_3$ in this table, even though these  are  also minimal quadratic configurations.  This is because if $Q$ denotes any of these configurations then $\forb(m,Q,F)=\forb(m,Q^c,F^c)$, which  is already included in Table~\ref{tab::results}.

\begin{center}
\captionof{table}{Minimal Simple Cubic Configurations with 4 Rows}\label{tab:mincub}
\begin{tabular}{ | c | c | c |  c| c|}
\hline 
 & Configuration  $F_i$ & Quadratic Const.(s) & Cubic Const.(s) & Proposition\\ 
\hline
$1_{4,1}$ & $\begin{bmatrix}1\\ 1\\  1\\ 1\end{bmatrix}$ & $I\times I$ & $I\times I\times I$ & Prop.~\ref{Const1}\\ 
\hline
$F_{9}$ & $\begin{bmatrix}1 & 0 & 0\\ 0 & 1 & 0 \\ 0 & 0 & 1\\ 0 & 0 & 1\end{bmatrix}$ & $\begin{matrix}I^c\times I^c\\ I^c\times T\\ T\times T\end{matrix}$ & $\begin{matrix}I^c\times I^c\times T\end{matrix}$ & Prop.~\ref{ConstF9}\\ 
\hline
$F_{10}$ & $\begin{bmatrix}1 & 0 & 0\\ 0 & 1 & 0 \\ 0 & 0 & 1\\ 0 & 0 & 0\end{bmatrix}$ &  $\begin{matrix}I^c\times I^c\\ I^c\times T\\ T\times T\end{matrix}$ & $\begin{matrix}I^c\times I^c\times T\end{matrix}$ & Prop.~\ref{ConstF9}\\ 
\hline
$F_{11}$ & $\begin{bmatrix}1 & 0 & 1 & 0\\ 1 & 0 & 0 & 1\\ 0 & 1 & 1 & 0\\0 & 1 & 0 & 1\end{bmatrix}$ & $\begin{matrix}I\times T\\ I^c\times T\\ T\times T\end{matrix}$ & $\begin{matrix}T\times T\times T\end{matrix}$ & Prop.~\ref{ConstF11}\\ 
\hline
$F_{12}$ & $\begin{bmatrix}1 &0 & 0 & 1\\ 0 & 1 & 0 & 1\\ 0 & 0 & 1 & 1\\ 1 & 1 & 1 & 0\end{bmatrix}$ & All & All & Lem.~\ref{L-all prod}\\ 
\hline
$F_{13}$ & $\begin{bmatrix}1 & 1 & 0 & 0\\ 0 & 1 & 1 & 0\\ 0 & 1 & 0 & 1\\ 0 & 0 & 1 & 1\end{bmatrix}$ & All & $\begin{matrix}T\times T \times T \end{matrix}$ & $\begin{matrix} \text{Lem.~\ref{L-all prod}}\\ \text{Prop.~\ref{ConstF13}}\end{matrix}$\\ 
\hline
\end{tabular}
\end{center}

In addition to this, the compliment of $1_{4,1}$ (which we denote by $0_{4,1}$), $F_9^c,\ F_{10}^c,\ $and $F_{12}^c$ are minimal simple cubic configurations, and the products avoiding these configurations are the complements of the products avoiding their complements.

Table~\ref{tab::results} contains the asymptotic values for all pairings of the configurations mentioned above when at least one of the configurations is cubic.  We note that all exact results stated below hold for $m$ sufficiently large.

\begin{center}
\captionof{table}{Results}\label{tab::results}
\begin{tabular}{ | c | c | c |  c| c | c | c |  c| c| c| c|}
\hline 
& $1_{4,1}$  & $F_9$ & $F_{10}$  & $F_{11}$ & $F_{12}$ & $F_{13}$&  $0_{4,1}$& $F_9^c$ & $F_{10}^c$& $F_{12}^c$ \\ 
\hline
$1_{3,1}$ & $\begin{matrix} \Theta(m^2)\\ \text{Rm }\ref{Rem}\end{matrix}$ & $\begin{matrix}m+2 \\ \text{Cr }\ref{1F9}\end{matrix}$ & $\begin{matrix}\Theta(1) \\ \text{Cr }\ref{1BB}\end{matrix}$  & $\begin{matrix}\Theta(m^{3/2}) \\ \text{Cr }\ref{1-1F11}\end{matrix}$ & $\begin{matrix}\Theta(m^2) \\ \text{Rm }\ref{Rem}\end{matrix}$  & $\begin{matrix} \Theta(m^2)\\ \text{Rm }\ref{Rem}\end{matrix}$ & $\begin{matrix} \Theta(1)\\ \text{Cr }\ref{1BB}\end{matrix}$ & $\begin{matrix} \Theta(m^2)\\ \text{Rm }\ref{Rem}\end{matrix}$ & $\begin{matrix}\Theta(m^2) \\ \text{Rm }\ref{Rem}\end{matrix}$& $\begin{matrix}\Theta(m^2) \\ \text{Rm }\ref{Rem}\end{matrix}$\\ 
\hline
$1_{2,2}$ & $\begin{matrix} \Theta(m^2)\\ \text{Rm }\ref{Rem}\end{matrix}$ & $\begin{matrix}m+3 \\ \text{Cr }\ref{1F9}\end{matrix}$ & $\begin{matrix}\Theta(1) \\ \text{Cr }\ref{1BB}\end{matrix}$  & $\begin{matrix}\Theta(m^{3/2}) \\ \text{Cr }\ref{1-2F11}\end{matrix}$ & $\begin{matrix}\Theta(m^2) \\ \text{Rm }\ref{Rem}\end{matrix}$  & $\begin{matrix} \Theta(m^2)\\ \text{Rm }\ref{Rem}\end{matrix}$ & $\begin{matrix} \Theta(1)\\ \text{Cr }\ref{1BB}\end{matrix}$ & $\begin{matrix} \Theta(m^2)\\ \text{Rm }\ref{Rem}\end{matrix}$ & $\begin{matrix}\Theta(m^2) \\ \text{Rm }\ref{Rem}\end{matrix}$& $\begin{matrix}\Theta(m^2) \\ \text{Rm }\ref{Rem}\end{matrix}$\\ 

\hline
$I_3$ & $\begin{matrix} \Theta(1)\\ \text{Cr }\ref{1BB}\end{matrix}$ & $\begin{matrix}\Theta(m^2) \\ \text{Rm }\ref{Rem}\end{matrix}$ & $\begin{matrix}\Theta(m^2) \\ \text{Rm }\ref{Rem}\end{matrix}$ & $\begin{matrix}\Theta(m^2) \\ \text{Rm }\ref{Rem}\end{matrix}$ & $\begin{matrix}\Theta(m^2) \\ \text{Rm }\ref{Rem}\end{matrix}$ & $\begin{matrix}\Theta(m^2) \\ \text{Rm }\ref{Rem}\end{matrix}$ & $\begin{matrix}\Theta(m^2) \\ \text{Rm }\ref{Rem}\end{matrix}$ & $\begin{matrix}\Theta(m^2) \\ \text{Rm }\ref{Rem}\end{matrix}$ & $\begin{matrix}\Theta(m^2) \\ \text{Rm }\ref{Rem}\end{matrix}$ & $\begin{matrix}\Theta(m^2) \\ \text{Rm }\ref{Rem}\end{matrix}$ \\ 
\hline
$Q_{3}$ & $\begin{matrix}\Theta(m) \\ \text{Cr }\ref{1Q3}\end{matrix}$  & $\begin{matrix} \Theta(m) \\ \text{Th }\ref{F9}\end{matrix}$  & $\begin{matrix} \Theta(m) \\ \text{Cr }\ref{1Q3}\end{matrix}$  & $\begin{matrix} \Theta(m^{3/2}) \\ \text{Cr }\ref{Q3F11} \\ \end{matrix}$ & $\begin{matrix}\Theta(m^2) \\ \text{Rm }\ref{Rem}\end{matrix}$ & $\begin{matrix} \Theta(m^2) \\ \text{Rm }\ref{Rem} \end{matrix}$& $\begin{matrix} \Theta(m) \\ \text{Cr }\ref{1Q3}\end{matrix}$ & $\begin{matrix} \Theta(m) \\ \text{Th }\ref{F9}\end{matrix}$& $\begin{matrix}\Theta(m)  \\ \text{Cr }\ref{1Q3}\end{matrix}$ & $\begin{matrix}\Theta(m^2) \\ \text{Rm }\ref{Rem} \end{matrix}$\\ 
\hline

$Q_{8}$ &  $\begin{matrix}\Theta(m) \\ \text{Pr }\ref{Q8}\end{matrix}$ & $\begin{matrix}\Theta(m^2) \\ \text{Rm }\ref{Rem}\end{matrix}$ & $\begin{matrix}\Theta(m^2) \\ \text{Rm }\ref{Rem}\end{matrix}$ & $\begin{matrix}\Theta(m^2) \\ \text{Rm }\ref{Rem}\end{matrix}$ & $\begin{matrix}\Theta(m^2) \\ \text{Rm }\ref{Rem}\end{matrix}$ & $\begin{matrix}\Theta(m^2) \\ \text{Rm }\ref{Rem}\end{matrix}$ & $\begin{matrix}\Theta(m) \\ \text{Pr }\ref{Q8}\end{matrix}$ & $\begin{matrix}\Theta(m^2) \\ \text{Rm }\ref{Rem}\end{matrix}$ & $\begin{matrix}\Theta(m^2) \\ \text{Rm }\ref{Rem}\end{matrix}$ & $\begin{matrix}\Theta(m^2) \\ \text{Rm }\ref{Rem}\end{matrix}$ \\ 
\hline

$Q_{9}$ &  $\begin{matrix}3m-2 \\ \text{Cr }\ref{Q9}\end{matrix}$ & $\begin{matrix}\Theta(m^2) \\ \text{Rm }\ref{Rem}\end{matrix}$ & $\begin{matrix}\Theta(m^2) \\ \text{Rm }\ref{Rem}\end{matrix}$ & $\begin{matrix}\Theta(m^2) \\ \text{Rm }\ref{Rem}\end{matrix}$ & $\begin{matrix}\Theta(m^2) \\ \text{Rm }\ref{Rem}\end{matrix}$ & $\begin{matrix}\Theta(m^2) \\ \text{Rm }\ref{Rem}\end{matrix}$ & $\begin{matrix}3m-2 \\ \text{Cr }\ref{Q9}\end{matrix}$ & $\begin{matrix}\Theta(m^2) \\ \text{Rm }\ref{Rem}\end{matrix}$ & $\begin{matrix}\Theta(m^2) \\ \text{Rm }\ref{Rem}\end{matrix}$ & $\begin{matrix}\Theta(m^2) \\ \text{Rm }\ref{Rem}\end{matrix}$ \\ 
\hline
\hline

$1_{4,1}$ &  & $\begin{matrix} m+5\\ \text{Cr }\ref{1F9}\end{matrix}$ & $\begin{matrix} \Theta(1)\\ \text{Cr }\ref{1BB}\end{matrix}$ & $\begin{matrix} \Theta(m^{3/2})\\ \text{Pr }\ref{Pr 1F11}\end{matrix}$ & $\begin{matrix} \Theta(m^3)\\ \text{Rm }\ref{Rem}\end{matrix}$ & $\begin{matrix} \Theta(m^2)\\ \text{Pr }\ref{P-F9 and10} \end{matrix}$ & $\begin{matrix} \Theta(1)\\ \text{Cr }\ref{1BB}\end{matrix}$ & $\begin{matrix} \Theta(m^3)\\ \text{Rm }\ref{Rem}\end{matrix}$  & $\begin{matrix} \Theta(m^3)\\ \text{Rm }\ref{Rem}\end{matrix}$  & $\begin{matrix} \Theta(m^3)\\ \text{Rm }\ref{Rem}\end{matrix}$  \\ 
\hline

$F_{9}$ & & & $\begin{matrix} \Theta(m^3)\\ \text{Rm }\ref{Rem}\end{matrix}$ & $\begin{matrix} \Theta(m^2)\\ \text{Pr }\ref{P-F9 and10} \end{matrix}$ & $\begin{matrix} \Theta(m^3)\\ \text{Rm }\ref{Rem}\end{matrix}$ & $\begin{matrix} \Theta(m^2)\\ \text{Pr }\ref{P-F9 and10} \end{matrix}$ & $\begin{matrix} \Theta(m^3)\\ \text{Rm }\ref{Rem}\end{matrix}$ & $\begin{matrix} \Theta(m^2)\\ \text{Pr }\ref{F9c} \end{matrix}$& $\begin{matrix} \Theta(m^2)\\ \text{Pr }\ref{F9c} \end{matrix}$ & $\begin{matrix} \Theta(m^3)\\ \text{Rm }\ref{Rem}\end{matrix}$ \\ 
\hline

$F_{10}$ & & &  & $\begin{matrix} \Theta(m^2)\\ \text{Pr }\ref{P-F9 and10} \end{matrix}$ & $\begin{matrix} \Theta(m^3)\\ \text{Rm }\ref{Rem}\end{matrix}$ & $\begin{matrix} \Theta(m^2)\\ \text{Pr }\ref{P-F9 and10} \end{matrix}$ & $\begin{matrix} \Theta(m^3)\\ \text{Rm }\ref{Rem}\end{matrix}$ & $\begin{matrix} \Theta(m^2)\\ \text{Pr }\ref{F9c} \end{matrix}$& $\begin{matrix} \Theta(m^2)\\ \text{Pr }\ref{F9c} \end{matrix}$ & $\begin{matrix} \Theta(m^3)\\ \text{Rm }\ref{Rem}\end{matrix}$ \\ 
\hline

$F_{11}$ & & &  & & $\begin{matrix} \Theta(m^3)\\ \text{Rm }\ref{Rem}\end{matrix}$ & $\begin{matrix} \Theta(m^3)\\ \text{Rm }\ref{Rem}\end{matrix}$ & $\begin{matrix} \Theta(m^{3/2})\\ \text{Pr }\ref{Pr 1F11} \end{matrix}$ & $\begin{matrix} \Theta(m^2)\\ \text{Pr }\ref{P-F9 and10} \end{matrix}$& $\begin{matrix} \Theta(m^2)\\ \text{Pr }\ref{P-F9 and10}\end{matrix}$ & $\begin{matrix} \Theta(m^3)\\ \text{Rm }\ref{Rem}\end{matrix}$ \\ 
\hline

$F_{12}$ & & &  & &  & $\begin{matrix} \Theta(m^3)\\ \text{Rm }\ref{Rem}\end{matrix}$ & $\begin{matrix} \Theta(m^3)\\ \text{Rm }\ref{Rem}\end{matrix}$ & $\begin{matrix} \Theta(m^3)\\ \text{Rm }\ref{Rem}\end{matrix}$& $\begin{matrix} \Theta(m^3)\\ \text{Rm }\ref{Rem}\end{matrix}$ & $\begin{matrix} \Theta(m^3)\\ \text{Rm }\ref{Rem}\end{matrix}$\\ 
\hline

$F_{13}$ & & &  & &  &  &$\begin{matrix} \Theta(m^2)\\ \text{Pr }\ref{P-F9 and10}\end{matrix}$ & $\begin{matrix} \Theta(m^2)\\ \text{Pr }\ref{P-F9 and10}\end{matrix}$& $\begin{matrix} \Theta(m^2)\\ \text{Pr }\ref{P-F9 and10}\end{matrix}$ & $\begin{matrix} \Theta(m^3)\\ \text{Rm }\ref{Rem}\end{matrix}$\\ 
\hline

\end{tabular}
\end{center}

In this section we determine all product constructions that avoid the minimal cubic configurations mentioned above, where we note that if a configuration $A$ is avoided by the product $B$ then $A^c$ is avoided by the product $B^c$.  We will then be able to obtain most of our lower bound results from the following observation:
\begin{rem}\label{Rem}
    If $F$ and $G$ are both avoided by the same $p$-fold product construction then $\forb(m,F,G)=\Omega(m^p)$.
\end{rem}
We note that proving $\forb(m,F,G)=\Omega(m^2)$ when either $F$ or $G$ is a minimal quadratic configuration implies that $\forb(m,F,G)=\Theta(m^2)$, and similarly if $\forb(m,F,G)=\Omega(m^3)$ for $F$ or $G$ a minimal cubic configuration then $\forb(m,F,G)=\Theta(m^3)$.

\begin{prop}\label{Const1}
    The only 2-fold product avoiding $1_{4,1}$ is $I\times I$.  The only 3-fold product avoiding $1_{4,1}$ is $I\times I\times I$. 
\end{prop}
\begin{proof}
    Note that  $1_{4,1}\prec I_5^c,T_5$, so any product using $I^c$ or $T$ will contain $1_{4,1}$.  There are only three 1's in each  column of $I\times I\times I$, so $1_{4,1}\nprec I\times I\times I$, and it follows that $1_{4,1}\nprec I\times I$ as well.
\end{proof}

\begin{lem}\label{L-F9}
     $F_9,F_{10},F_9^c,F_{10}^c\prec [01]\times [01]\times T_4$.
\end{lem}
\begin{proof}
    The last two rows of $F_9,F_{10},F_9^c,F_{10}^c$ are contained in $T_4$, and hence the last three rows of these configurations will be contained in $[01]\times T_{4}$ and all of the configurations will be contained in $[01]\times [01]\times T_{4}$.
\end{proof}

\begin{prop}\label{ConstF9}
    $F_9$ and $F_{10}$ are avoided by every 2-fold product not involving $I$, and they are contained in every 2-fold product involving $I$.  The only 3-fold product avoiding $F_9$ and $F_{10}$ is $I^c\times I^c\times I^c$.
\end{prop}
\begin{proof}
    Note that $I_3$ is avoided by every 2-fold product not involving $I$ by \cite{AK}, and because $I_3\prec F_9,F_{10}$ it follows that these products must also avoid $F_9$ and $F_{10}$.  Observe that $F_9,F_{10}\prec [01]\times I_3$, and hence $F_9$ and $F_{10}$ will be contained in any 2-fold product involving $I$.  It follows from Lemma \ref{L-F9} that $F_9,F_{10}$ will be contained in any 3-fold product involving $T$, so the only 3-fold product that can avoid these configurations is $I^c\times I^c\times I^c$, and \cite{survey} notes that  this is indeed the case.
\end{proof}

\begin{lem}\label{L-F11}
    $F_{11},F_{13}\prec [01]\times [01]\times I_2=[01]\times [01]\times I_2^c$.
\end{lem}
\begin{proof}
    $F_{11}=I_2\times I_2\prec [01]\times [01]\times I_2$.  The second and  third rows of $F_{13}$ are equal to $[01]\times [01]$, and the remaining rows consist of columns of $I_2$.  We thus have $F_{13}\prec [01]\times[01]\times I_2$.
\end{proof}
\begin{prop}\label{ConstF11}
    $F_{11}\nprec I\times T,I^c\times T,  T\times T$  and it is contained in all other 2-fold products.  The only 3-fold product that avoids $F_{11}$ is $T\times T\times T$.
\end{prop}
\begin{proof}
    Note that  $Q_9\prec F_{11}$ and that $Q_9\nprec I\times T,I^c\times T$, so it follows that this is also the case for $F_{11}$.  Because $F_{11}=I_2\times I_2$ and $I_2\prec I,I^c$, it follows that every 2-fold product consisting only of $I$'s and $I^c$'s contains $F_{11}$.  \cite{survey} notes that $F_{11}\nprec T\times T\times T$, so it also follows that $F_{11}\nprec T\times T$.  It follows from Lemma \ref{L-F11} that every 3-fold product involving an $I$ or $I^c$ contains $F_{11}$, so the only 3-fold product that can avoid $F_{11}$ is $T\times T\times T$.
\end{proof}

\begin{lem}\label{L-all prod}
    All 2-fold products of $I,\ I^c$ and $T$ avoid $F_{13}$.  All 3-fold products avoid $F_{12}$ and $F_{12}^c$
\end{lem}
\begin{proof}
    Every two rows of the first three rows of $F_{13}$ contains $\begin{bmatrix}
    1 & 0 & 0 & 1\\ 
    0 & 1 & 0 & 1
    \end{bmatrix}$, and as no two rows of $I,\ I^c,$ or $T$ contains this configuration, the first three rows of $F_{13}$ can not be found in any 2-fold product of these matrices.  Any two rows of $F_{12}$ contains $\begin{bmatrix}0 & 1 & 1\\ 1 & 0 & 1\end{bmatrix}$, which again is contained in no two rows of $I,\ I^c$ or $T$, so this can not be found in any 3-fold product of these matrices.  Similar logic holds for $F_{12}^c$.
\end{proof}

\begin{prop}\label{ConstF13}
    The only 3-fold product that avoids $F_{13}$ is $T\times T\times T$.
\end{prop}
\begin{proof}
    By Lemma \ref{L-F11} every 3-product involving $I$ or $I^c$ contains $F_{13}$, and \cite{survey} notes that $F_{13}\nprec T\times T\times T$. 
\end{proof}

\iffalse
\begin{cor}
    $\forb(m,\set{I_3,0_{4,1}})=\Theta(m^2)$.
\end{cor}
\begin{proof}
    Use $I^c\times I^c$.
\end{proof}

\begin{cor}
    $\forb(m,\set{1_{3,1},1_{2,2},F_9^c,F_{10}^c})=\Theta(m^2)$.
\end{cor}
\begin{proof}
    Use the construction $I\times I$.
\end{proof}
\begin{cor}
    $\forb(m,\set{I_3,Q_8,F_9,F_{10},F_{11},F_9^c,F_{10}^c})=\Theta(m^2)$.
\end{cor}
\begin{proof}
    Use the construction $T\times T$.
\end{proof}
\begin{cor}
    $\forb(m,\set{Q_9,F_9,F_{10}})=\forb(m,\set{Q_9,F^c_9,F^c_{10}})=\Theta(m^2)$.
\end{cor}
\begin{proof}
    Use $I^c\times T$ for the first bound, and because $Q_9^c=Q_9$ $I\times T$ gives the second bound.
\end{proof}

\begin{cor}
    $\forb(m,Q_9,F_{11})=\Theta(m^2)$.
\end{cor}
\begin{proof}
    Use $I\times T$.
\end{proof}

\begin{cor}
    $\forb(m,\set{F,F_{12},F_{13},F_{12}^c})=\Omega(m^2)$ for $F$ any minimal quadratic or cubic configuration.  $\forb(m,\set{F,F_{12},F_{12}^c})=\Theta(m^3)$ for $F$ any minimal cubic configuration.
\end{cor}
\fi

\section{Inductive Results}

In this section we prove a variety of upper bounds by using two standard techniques: Theorem~\ref{thm:BB} and the following standard induction method.
Let $F$ be a $k$-rowed matrix. Suppose we have $A\in\Avoid(m,F)$ such that $|A|=\forb(m,F)$. Consider deleting a row $r$. Let $C_r(A)$ be the matrix that consists of the repeated columns of the matrix that is obtained when deleting row $r$ from $A$. If we permute the rows of $A$ so that $r$ becomes the first row, then after some column permutations, $A$ looks like this:
\[
    A =
    \begin{matrix}
        r \\ \\
    \end{matrix}
    \begin{bmatrix}
        0 & \cdots & 0 & 1 & \cdots & 1\\
        B_r(A) & & C_r(A) & C_r(A) & & D_r(A)\\
    \end{bmatrix}.
        \label{sdecomp}
\]
where $B_r(A)$ are the columns that appear with a $0$ on row $r$,
but don't appear with a $1$, and $D_r(A)$ are the columns that appear with a
$1$ but not a $0$. We have that
\[\forb(m,F) \leq |C_r(A)| +
      \forb(m-1,F),
      \]
as $[B_r(A)C_r(A)D_r(A)]\in\Avoid(m-1,F)$. This is used usually in the form that if $F\prec [01]\times F'$, then 
\[\forb(m,F) \le \forb(m-1,F')+
      \forb(m-1,F).
      \]

We let $1_{k,\ell}$ denote the $k\times \ell$ matrix where every entry is 1.  Similarly, we define $0_{k,\ell}$ to be the $k\times \ell$ matrix where every entry is 0.  We use the notation $C_r:=C_r(A)$ when it is clear from context what the underlying matrix $A$ is.

\begin{prop}\label{Q8}
    $\forb(m,Q_8,1_{k,\ell})=\forb(m,Q_8,0_{k,\ell})=\Theta(m)$.
\end{prop}
\begin{proof}
    As $Q_8^c=Q_8$ we see that these two values are equal, so we only address the $1_{k,\ell}$ case.  Note that $I_m$ gives the lower bound.  For the upper bound, note that $Q_8=[01]\times I_2$.  It follows that when we apply the standard induction that $C_r$ can not contain $I_2=I_2^c$.  But by Theorem~\ref{thm:BB} if $|C_r|>BB(k+\ell)$ we must have $T_{k+\ell}\prec C_r$, which would contradict $1_{k,\ell}\nprec A$.  Thus we must have $|C_r|\le BB(k+\ell)$, so we can inductively assume a linear bound for $\forb(m,Q_8,1_{k,\ell})$.
\end{proof}

\begin{lem}\label{L-quad}
    $\forb(m,[01]\times[01]\times I_r,[01]\times[01]\times I_r^c,[01]\times[01]\times T_r)=O(m^2)$.
\end{lem}
\begin{proof}
    By using the standard induction and Theorem~\ref{thm:BB} one gets that  $\forb(m,[01]\times I_r,[01]\times I_r^c,[01]\times T_r)=O(m)$.  Given this, when we apply the standard induction for $\forb(m,[01]\times[01]\times I_r,[01]\times[01]\times I_r^c,[01]\times[01]\times T_r)$ we get a quadratic upper bound.
\end{proof}

\begin{prop}\label{P-F9 and10}
    $\forb(m,F,G)=O(m^2)$ for $F=1_{4,1},F_9,F_{10},F_9^c$, or $F_{10}^c$ and $G=F_{11}$ or $F_{13}$.
\end{prop}
In Table~\ref{tab::results} Proposition~\ref{P-F9 and10} is frequently quoted to prove $\Theta$ bounds. This is done so when common quadratic lower bound exists for $F$ and $G$ by product constructions listed in Table~\ref{tab:mincub}.
\begin{proof}
    This follows from Lemma~\ref{L-quad}, along with the observations that  $1_{4,1}\prec [01]\times [01]\times T_4$, $F_9,F_{10}\prec [01]\times [01]\times T_4$ by Lemma~\ref{L-F9}, and $F_{11},F_{13}\prec [01]\times[01]\times I_2$ by Lemma~\ref{L-F11}.
\end{proof}

\begin{prop}\label{F9c}
    $\forb(m,F,G)=\Theta(m^2)$ where $F=F_9$ or $F_{10}$ and $G=F_9^c$ or $F_{10}^c$.
\end{prop}
\begin{proof}
    The lower bound follows from the construction $T\times T$, and the upper bound is a consequence of Lemma~\ref{L-quad} and the observations that $F_9,F_{10},F_9^c,F_{10}^c\prec [01]\times [01]\times T_4$,\  $F_9,F_{10}\prec [01]\times I_3\prec [01]\times[01]\times I_3$ and $F_9^c,F_{10}^c\prec [01]\times I_3^c\prec [01]\times [01]\times I_3^c$.
\end{proof}
\section{Avoiding $Q_3(t)$}\label{S-F2}
We consider a slight generalization of $Q_3$  \vspace{5 mm}
\[
    Q_3(t)=\begin{bmatrix}
    0 &  \bovermat{t}{1 \cdots 1} & \bovermat{t}{0  \cdots 0}& 1\\ 
    0 &  0 \cdots 0 & 1 \cdots 1 &1
    \end{bmatrix},
\]
where we always assume $t\ge 2$ when we write $Q_3(t)$.  We have the following result from \cite{AK}.  
\begin{thm}\label{T-fam}
    $\forb(m,Q_3(t),t\cdot I_k)=\forb(m,Q_3(t),t\cdot I^c_k)=\Theta(m)$ for any fixed $k$.
\end{thm}
\begin{cor}\label{1Q3}
    $\forb(m,Q_3,F)=\Theta(m)$ for $F=1_{4,1},\ F_{10},\ 0_{4,1},\ F_{10}^c$.
\end{cor}
\begin{proof}
    Each of these $F$ is contained in either $I_{k}$ or $I_{k}^c$ for sufficiently large $k$, so Theorem~\ref{T-fam} gives the upper bound, and either $I_m$ or $I_m^c$ gives the lower bound.
\end{proof}

Our main result for this section will be a stability theorem which says that large $Q_3(t)$ avoiding matrices ``look like'' $I\times I^c$, and from this we will be able to prove an upper bound for $\forb(m,Q_3,F_{11})$, and more generally for $\forb(m,Q_3(t),I_r\times I_s)$.  We first introduce some terminology for the proof.

We will say that a row $r$ is \de{sparse} when restricted to a set of columns $C$ if, restricted to $C$, $r$ has at least one 0 but fewer than $t$ 0's (i.e. $r$ has few 0's but is not identically 1), and we will say that a row $r$ is \de{dense} when restricted to a set of columns $C$ if $r$ has at least one 1 and at least $t$ 0's within the columns of $C$ (i.e. $r$ has many 0's but is not identically 0).  We will say that a column $c\in C$ is \de{identified} by a sparse row $r$ if $r$ has a 0 in column $c$.

If $A$ is a matrix and $C$ is a set of columns (not necessarily a subset of the columns of $A$), then $A\setminus C$ will denote the set of columns in $A$ that are not in $C$.  We define the matrix $Q_3(t;0)$ to be $Q_3(t)$ without its column of 1's.  Lastly, we restate Theorem~\ref{T-fam} as follows:   for any fixed $k$ and $t$ there exists a constant $c_{k,t}$ such that if $A$ is an $m$-rowed simple matrix with $|A|> c_{k,t}m$ and $Q_3(t)\nprec A$, then $t\cdot I_k\prec A$.

\begin{thm}\label{T-Q3Main}
Let $A\in \Avoid(m,Q_3(t))$ with $|A|=\omega(m\log m)$.  There exists a set of integers $\set{k_1,\ldots,k_y}$ and a set $A'=\set{A'_1,\ldots,A'_y},$ of configurations $A'_j\prec A$ such that:
\begin{enumerate}
    \item $k_{j+1}\le \half k_j$ for all $j$, and $y\le\log m$.
    \item There exists $k_j$ rows of $A$ such that the columns of $A'_j$ restricted to these rows are columns of $I_{k_j}$.
    \item If $i$ is a column of $I_{k_j}$ and $C_i^j$ is the set of columns in $A'_j$ that are an $i$ column in the rows mentioned above, then no row restricted to $C_i^j$ is dense, and every column of $C_i^j$ is identified by some sparse row.
    \item $|A|=\Theta(\sum |A'_j|)$.
\end{enumerate}
\end{thm}
We first present an outline of the proof before going into the details. We are given a large $Q_3(t)$ avoiding matrix $A_0$, and as a first step we remove all rows from $A_0$ that have few 1's (for technical reasons) to get a new matrix $A_1$.  We then find the largest $t\cdot I_k$ in $A_1$, and our goal is to use this as the $I_{k_1}$ base for $A'_1$.  To do so, we trim $A_1$ by getting rid of all columns of $C_i^1$ that are not identified by a sparse row, as well as all rows that are dense restricted to some $C_i^1$.  This gives us $A'_1$, and we repeat the process on the remaining columns of $A_1,\ A_2$ (after again removing rows with few 1's). It turns out that the largest $t\cdot I$ in $A_2$, $I_{k_2}$,  will satisfy $k_2\le \half k_1$, and thus we can repeat this process at most $\log m$ times.  At each step we remove only $O(m)$  columns, so in total only $O(m\log m)$ columns of $A_0$ were removed.  As $|A_0|=\omega(m\log m)$, the columns that remain (those of $A'$) must be asymptotically as large as our original $A_0$.
\begin{proof}
Let $A_0\in \Avoid(m,Q_3(t))$ with $|A_0|=\omega(m\log m)$.  Let $R_1$ denote the set of rows of $A_0$ that have fewer than $3t-2$ 1's, and let $A_1$ denote $A_0$ with these rows removed.  Note that $A_1$ need not be a simple matrix, but if $C_{R_1}$ denotes the set of columns that have a 1 in some row of $R_1$, then $A_1\setminus C_{R_1}$ will be simple.  As $|C_{R_1}|\le (3t-2)m=O(m),\ |A_1\setminus C_{R_1}|=\Theta(|A_0|)$.  Note that we will be working with the matrix $A_1$, \textit{not} its simplification $A_1\setminus C_{R_1}$, in order to use the fact that every row has at least $3t-2$ 1's.

Define $k_1$ to be the largest integer such that $t\cdot I_{k_1}\prec A_1$.  As $|A_1\setminus C_1|=\omega(m)$, Theorem~\ref{T-fam} tells us that we have $t\cdot I_k\prec A_1\setminus C_1\prec A_1$ for any fixed $k$ (so in particular we can assume that $k_1\ge 3$).  Rearrange rows so that this $t\cdot I_{k_1}$ appears in the first $k_1$ rows of $A_1$.

Note that no column of $A_1$ can have two 1's in the first $k_1$ rows.  Indeed, any two rows of $t\cdot I_{k_1}$ for $k_1\ge 3$ induce a $Q_3(t;0)$, and hence if a column had 1's in two of these rows we would have $Q_3(t)\prec A_1$.  We can thus partition the columns of $A_1$ as follows.  We will say that a column $c$ belongs to the set $C_i^1$ for $1\le i\le k_1$ if $c$ has a 1 in row $i$, and we will say that $c\in C^2$ if $c$ has no 1's in these rows.  We will make the additional assumption that the $t\cdot I_{k_1}$ we placed in the first $k_1$ rows was such that $|C^2|$ is minimal.  Note that $|C_i^1|\ge 3t-2$ for all $i$, as otherwise the $i$th row would belong to $R_1$ and hence not be in $A_1$.  

We now examine the rows that are dense in some $C_i^1$.

\begin{lem}\label{L-dense}
    If a row $r$ restricted to $C_i^1$ is dense, then restricted to $A_1\setminus C_i^1$, $r$ has at most $t-1$ 1's or $r$ is identically 1.
\end{lem}
\begin{proof}
    Assume $r$ is dense restricted to $C_i^1$, i.e. it has at least $t$ 0's and one 1 restricted to $C_i^1$.  If $r$ had $t$ 1's and a 0 in $A\setminus C_i^1$, then by looking at the $i$th row, row $r$, and the relevant columns, we would find a $Q_3(t)$.
\end{proof}
We would like to strengthen the above lemma to say that dense rows are either identically 0 or identically 1 outside of their $C_i^1$, and to do so we'll have to ignore a small number of columns of $A_1$.  We will say that a column $c$ is ``bad'' if there exists a row $r$ and integer $i$ such that $r$ is dense restricted to $C_i^1$, $r$ is not identically 1 in $A\setminus C_i^1$, and $c$ has a 1 in row $r$.  Let $\overline{C^1}$ denote the set of bad columns.

\begin{lem}
    $| \overline{C^1}|=O(m)$.
\end{lem}
\begin{proof}
	Each dense row $r$ contributes at most $t-1$ columns to $ \overline{C^1}$ by Lemma \ref{L-dense}, and hence $|\overline{C^1}|\le (t-1)m=O(m)$.
\end{proof}

We now wish to ignore the dense rows of $A_1$, as well as any rows of $\bigcup C_i^1$ that are not identified by a sparse row.  Rearrange rows so that the bottom $\ell$ rows of $A_1$ consist of all rows that when restricted to some $C_i^1$ are dense.  Let $\dense{C_i^1}$ denote the columns of $C_i^1$ that are not identified by a sparse row and that are not in $C_{R_1}$ or $\overline{C^1}$.  Let $\dense{A_1}$ denote $A_1$ restricted to the top $k_1$ rows, the bottom $\ell$ rows, and the columns of $\bigcup \dense{C_i^1}$.

\begin{lem}\label{L-simple}
	 $\dense{A_1}$ is a simple matrix.
\end{lem}
\begin{proof}
	
	Let $\hat{c}$ and $\hat{d}$ be columns of $\dense{A_1}$ with corresponding columns $c,d$ in $A_1\setminus C_{R_1}$ (as no $\dense{C_i^1}$ columns are in $C_{R_1}$).  If $\hat{c}=\hat{d}$, then clearly we must have $c,d\in  C_i^1$ for some $i$.  As $c\ne d$ (because  $A_1\setminus C_{R_1}$ is a simple matrix), we must have $c$ and $d$ differing in some row $r$ above the bottom $\ell$ rows, say $c$  has a 0 in row $r$ and $d$ has a 1.  But this means that  $r$ must be sparse (as every row between the top $k_1$ rows and bottom $\ell$ rows is either identically 0, identically 1, or sparse), and hence $c$ is identified by a sparse row, contradicting $\hat{c}$ belonging to $\dense{A_1}$.
\end{proof}

\begin{lem}\label{L-few dense}
	$|\dense{A_1}|=O(m)$.
\end{lem}
\begin{proof}
	By Lemma \ref{L-dense} (and the fact that $\dense{A_1}$ contains no columns of $\overline{C^1}$), we know that each row $r$ restricted to $\dense{C_i^1}$ can be one of four types: $r$ can be identically 0 restricted to $A_1\setminus C_i^1$ (in which case we will say it is a row of $B_{i,0}$), $r$ can be identically 1 restricted to $A_1\setminus C_i^1$ (in which case we will say it is a row of $B_{i,1}$), or $r$ can itself be either identically 0 or identically 1.  We thus have that the matrix $B_i$ formed by restricting $\dense{A_1}$ to the columns $\dense{C_i^1}$ and to the rows of $B_{i,0}$ and $B_{i,1}$ is simple with $|\dense{C_i^1}|$ columns.  Let $b_i$ denote the number of rows in $B_i$.  
	
    If $|B_i|>c_{3,t} b_i$, then we must have $t\cdot I_3\prec B_i$, and hence either $B_{i,0}$ or $B_{i,1}$ must contain a $Q_3(t;0)$.  If $B_{i,1}$ contains a $Q_3(t;0)$, then these rows and columns together with any column of $A_1\setminus C_i^1$ gives a $Q_3(t)$.  If $B_{i,0}$ contains a $Q_3(t;0)$, then one can find a $t\cdot I_{k_1+1}$ in $A_1$. Indeed, in $A_1$ (note that we are no longer ignoring the columns of $\overline{C^1}$ and $C_{R_1}$), take the two rows from $B_{i,0}$ that contain a $Q_3(t;0)$, ignore the at most $2t-2$ columns that have 1's in these rows outside of $C_i^1$, and swap these rows with rows $i$ and $k_1+1$.  After performing these steps, no column of $A_1$ has two 1's in any of the first $k_1+1$ rows (since we removed the at most $2t-2$ columns that could pose a problem), rows $i$ and $k_1+1$ by assumption have at least $t$ 1's, and as every other row had at least $3t-2$ 1's before ignoring the at most $2t-2$ columns, they all still have at least $t$ 1's.  Hence we have $t\cdot I_{k_1+1}\prec A_1$, contradicting our definition of $k_1$.  Thus we must have $|B_i|=|\dense{C_i}|\le c_{3,t} b_i$, and in total we have \[|\dense{A_1}|=\sum |\dense{C_i^1}|\le \sum c_t b_i\le c_t \ell\le c_t m,\] proving the statement.
\end{proof}

We now let $A_1'$ be $\bigcup C_i^1$ after removing the columns of $\dense{A_1},\ C_{R_1},\ $and $\overline{C^1}$ (which in total are only of size $O(m)$), along with the bottom $\ell$ rows.  If $|C^2|=O(m\log m)$, then $A'=\set{A'_1}$ meets all of the conditions of the theorem.  Otherwise we can repeat our argument.

Let $R_2$ denote the set of rows below the first $k_1$ rows such that if $r\in R_2$ then $r$ has fewer than $3t-2$ 1's when restricted to $C^2$, and let $C_{R_2}$ be the set of columns where one of these rows has a 1 in $C^2$.  Let $A_2$ be $A_1$ restricted to $C^2$ after ignoring the rows of $R_2$ and let $k_2$ be the largest integer such that $t \cdot I_{k_2}\prec A_2$.  Note that we can assume $k_2\ge 3$.

\begin{lem}\label{L-shrink}
	$k_2\le \half k_1$.
\end{lem}
\begin{proof}
	Note that any row $r$ that is part of this $t\cdot I_{k_2}$ must appear above the bottom $\ell$ rows (as restricted to $C^2$ the bottom $\ell$ rows either have fewer than $t$ 1's or they are identically 1).  Thus restricted to any $C_i^1$, $r$ is either identically 0, identically 1 or sparse.  We will say that a row $r$ is ``mostly 1'' restricted to $C_i^1$ if $r$  is identically 1 or sparse restricted to $C_i^1$ (i.e. $r$ has fewer than $t$ 0's restricted to these columns).  Rearrange rows so that this $t\cdot I_{k_2}$ appears in the first $k_2$ rows.  
	
	Note that because $k_2\ge 3$, no column can have two 1's in the first $k_2$ rows.  As $|C_i^1|\ge 3t-2\ge 2t-1$ for all $i$, any two rows that are mostly 1 restricted to any $C_i^1$ must contain a column with 1's in both of these rows. Hence restricted to any $C_i^1$ and the first $k_2$ rows, there can be at most one mostly 1 row.
	
	If row $1\le j\le k_2$ is not mostly 1 when restricted to any $C_i^1$, then we could use row $j$ to create a $t\cdot I_{k_1+1}\prec A_1$ by swapping it with our original $k_1+1$th row, contradicting the definition of $k_1$.  If there is precisely one $i$ such that  $j$ restricted to $C_i^1$ is mostly 1, then swapping row $j$ with the original  $i$th row gives a $t\cdot I_{k_1}$ that would have given us a smaller value for $|C^2|$ (as at least $3t-2$ 1's get added from $C^2$ and at most $t-1$ 1's are replaced by 0's of the mostly 1 row), which contradicts our choice of $t\cdot I_{k_1}\prec A_1$.  Hence every row $1\le j\le k_2$ must be mostly 1 restricted  to at least two different $C_i^1$, but as each $C_i^1$ can only contribute at most one mostly 1 row we must have $k_2\le  \half k_1$.
\end{proof}

We then perform identical arguments for the corresponding $C_i^2$ columns as we did with the $C_i^1$ columns to get an $A'_2$.  If $C^3$ is defined analogous to $C^2$ and if $C^3=O(m\log m)$, then we can take $A'=\set{A'_1,A'_2}$ which satisfies all the conditions of the theorems.  If not, we repeat the same argument.  But by Lemma \ref{L-shrink} this process can continue at most $\log m$ times, and when the process terminates $A'$ excludes only $O(m\log m)$ columns of $A_0$ (as it ignores $O(m)$ columns at each of the potentially $\log m$ steps), so it meets all of the criteria of the theorem.
\end{proof}

Theorem~\ref{T-Q3Main} allows us to reduce computing upper bounds of matrices in $\Avoid(m,\mathcal{F})$ where $Q_3(t)\in\mathcal{F}$ to computing upper bounds of matrices that are of the same form as the $A'_j$ matrices.

\begin{cor}\label{C-log bound}
    For $\mathcal{F}$ with $Q_3(t)\in \mathcal{F}$, let $\tilde{A}$ be the largest matrix such that $\tilde{A}\in \Avoid(m,\mathcal{F})$ and such that it meets all the requirements of the $A'_j$ matrices in the statement of Theorem~\ref{T-Q3Main}.  Then $\forb(m,\mathcal{F})=O(\max\set{|\tilde{A}|,m}\log m)$.
\end{cor}
\begin{proof}
    The statement certainly holds if $\forb(m,\mathcal{F})=O(m\log m)$.  Assume $\forb(m,\mathcal{F})=\omega(m\log m)$.  Then if $A$ is a maximum sized matrix in $\Avoid(m,\mathcal{F})$ we can apply Theorem~\ref{T-Q3Main} to get a set of configurations $A'=\set{A_j'}$ with $|A_j'|\le |\tilde{A}|$ for all $j$ (as necessarily $A'_j\in \Avoid(m,\mathcal{F})$ since $A'_j\prec A\in \Avoid(m,\mathcal{F})$), and we have $|A|=O(\sum |A'_j|)$ or $|A|=O(|\tilde{A}|\log m)$.
\end{proof}

We suspect that the statement of Corollary~\ref{C-log bound} can be strengthened to $O(\max\set{|\tilde{A}|,m})$, but as stated the Corollary can still be used to prove near optimal results.  It is possible to get tighter upper bounds for certain configurations by using some of the additional structure provided by Theorem~\ref{T-Q3Main}.

\begin{thm}\label{T-Q3Upper}
    If $s\le r$ then $\forb(m,Q_3(t),I_r\times I_s^c)=O(m^{2-1/s})$.
\end{thm}
%Note that $\forb(m,Q_3(t),I_s\times I_r^c)$ has the same asymptotic bounds, as the 0-1 compliment of any matrix in $ \Avoid(m,Q_3(t),I_r\times I_s^c)$ is a matrix in $\Avoid(m,Q_3(t),I_s\times I_r^c)$.

\begin{proof}
    We first prove this for the case $t=2$.  Let $A\in \Avoid(m,Q_3(2),I_r\times I_s^c)$ with $|A|=\omega(m\log m)$ and let $A'$ be the corresponding set obtained from Theorem \ref{T-Q3Main}.   We focus our attention on bounding $|A'_1|$.  Note that restricted to $C_i^1$, there must exist $|C_i^1|$ rows that are distinct rows of $I_{|C_i^1|}^c$ (one to identify each column of $C_i^1$).  Denote a set of such rows by $R_i$.  If there exists a set of integers $\set{i_1,\ldots,i_r}$ such that $|R_{i_1}\cap \cdots\cap R_{i_r}|\ge s$, then by taking these $s$ rows, the rows $i_1,\ldots,i_r$ and the relevant columns we can find an $I_r\times I_s^c$ in $A'_1$ (since we have an $I_s^c$ occurring simultaneously under $r$ different $I_{k_1}$ columns).  How large can $|A_1'|=\sum |C_i^1|$ be given this restriction?
    
    We rephrase this problem in terms of graph theory.  We form a bipartite graph $G(C,R)$ where $v_i\in C$ for $1\le i\le k_1$ corresponding to the $C_i^1$ columns, and $r\in R$ corresponding to each row below the first $k_1$ rows.  $G$ will contain the edge $v_ir$ iff $r\in R_i$.  Our restriction of no set $\set{i_1,\ldots,i_r}$ such that $|R_{i_1}\cap \cdots\cap R_{i_r}|\ge s$ means that $G$ does not contain a $K_{r,s}$, the complete bipartite graph with vertex sets of size $r$ and $s$, with the $r$ vertices coming from $C$ and the $s$ vertices coming from $R$.  Using standard arguments from extremal graph theory, this graph can have at most $c|R||C|^{1-1/s}+d|C|\le cmk_1^{1/s}+dk_1$ edges for some constants $c$ and $d$.  Hence in total we have that
    \[
        \sum |A'_i|\le \sum (cm k_i^{1-1/s}+dk_i)\le cm k_1^{1-1/s}\sum \left(\frac{1}{2}\right)^{i(1-1/s)}+dk_1\sum \left(\half\right)^i=O(m^{2-1/s}), 
    \]
    and thus this is an asymptotic upper bound for $|A|=\Theta(\sum |A_i'|)$.
    
	We wish to generalize this argument for arbitrary $t$.  The key idea is that for each set $C_i^j$ we must find a set of rows $R_i^j$ with $|R_i^j|=\Theta_t(|C_i^j|)$ and such that $R_i^j$ contains an $I_{|R_i^j|}^c$.  Once we have this, we can perform the same graph argument on these $R_i^j$ rows as we did for the $R_i$ rows above and get the same asymptotic results.  The following lemma accomplishes this goal by taking $B=C_i^j$ after ignoring rows that are identically 0.

\end{proof}
\begin{lem}\label{L-avoiding rows}
	Given an integer $t$, let $B$ be a matrix consisting of rows with fewer than $t$ 0's such that every column of $B$ has a 0 in some row.  Then there exists a set of rows $R$ of $B$ such that:
	\begin{enumerate}
		\item $R$ contains an $I_{|R|}^c$.
		\item $|R|\ge 2^{2-t}|B|$.
	\end{enumerate}
\end{lem}
\begin{proof}
	The $t=2$ case is obvious (for every column take a row that has a 0 in the column), so inductively assume the statement holds up to $t-1$.  We wish to partition the columns of $B$ into two sets, $B_1$ and $B_2$.  Remove the leftmost column $c$ of $B$ and add it to $B_1$, and remove all columns $c'$ of $B$ where there exists a row $r$ such that $r$ has a 0 in both column $c$ and column $c'$ and add these columns to $B_2$.  Repeat this process until every column of $B$ is in one of these sets, and note that $B_i\ge \half |B|$ for some $i$.  Note that as every column of $B$ was identified, every column of $B_1$ and $B_2$ is also identified.
	
	If $B_1\ge \half |B|$, then note that no row $r$ has more than one 0 in $B_1$ (if $r$ had 0's in $c,c'\in B_1$ with $c$ to the left of $c'$, then $c'$ should have been added to $B_2$), so by the $t=2$ case we can find a set $R$ with $|R|=|B_1|\ge \half |B|$ that contains an $I_{|R|}^c$.
	
	If $|B_2|\ge \half |B|$, then note that $B_2$'s rows all have at most $t-2$ 0's (as every row with a 0 in some $c'$ originally had a 0 in the corresponding $c$ column from $B_1$), so by the inductive hypothesis we can find a set $R$ with $|R|\ge 2^{2-(t-1)}|B_2|\ge2^{2-t}|B|$ that contains an $I_{|R|}^c$.
\end{proof}

We can use the graph idea from the proof of Theorem~\ref{T-Q3Upper} to achieve lower bounds as well.
\begin{thm}\label{T-Q3Lower}
    $\forb(m,Q_3(t),I_r\times I_s^c)=\Omega(ex(m,K_{r,s}))$.
\end{thm}
\begin{proof}
    We define a generalized product operation for matrices. Let $A$ and $B$ be simple matrices with $m_1$ and $m_2$ rows respectively and $G=G(C_A,C_B)$ a bipartite graph with the vertex set $C_A$ corresponding to the set of columns of $A$ and $C_B$ to the set of columns of $B$.  We define $A\times_G B$ to be the simple matrix on $m_1+m_2$ rows such that it contains the column defined by placing the column $a\in C_A$ on the column $b\in C_B$ iff $ab\in E(G)$.  Thus $|A\times_G B|=|E(G)|$.
    
    Let $G(V,W)$ be a bipartite graph  on $m$ vertices such that $G$ avoids $K_{r,s}$ and such that  $G$ has the maximum number of edges.  Note that using the probabilistic method it is easy to show that $|E(G)|\ge \half ex(m,K_{r,s})$.  We claim that $A=I_{|V|}\times_G I^c_{|W|}\in \Avoid(m,Q_3(t),I_r\times I_s^c)$, and hence $\forb(m,Q_3(t),I_r\times I_s^c)\ge \half ex(m,K_{r,s})$.  We certainly have $Q_3(t)\nprec A$ as $A$ is a sub-matrix of $I_{a}\times I_a^c$ for $a=\max\set{|V|,|W|}$, which avoids $Q_3(t)$.  Note that if $I_r\times I_s^c\prec A$ Then we must have all of the $I_r$ rows coming entirely from either the $I_{|V|}$ rows of $A$ or the $I_{|W|}^c$ rows and the $I_s^c$ rows coming entirely from the other.  Indeed, no two rows of the $I_{|V|}$ block of $A$ contains a column of two 1's, but every row of $I_r$ in $I_r\times I_s^c$ together with a row of $I_s^c$ contains a column of two 1's, so the $I_{|V|}$ rows can contribute to at most one of these blocks. Further note that if $s\ge 3$ then the $I_s^c$ must come from the $I_{|W|}^c$ block (as it needs a column with two 1's), and similarly if $r\ge 3$ then $I_r$ must come from the $I_{|V|}$ block (and hence again the $I_s^c$ must come from the $I_{|W|}^c$ block).  
    
    Now consider $B=I_{|V|}\times_G I_{|W|}$.  If $I_r\times I_s^c\prec A$ then we certainly have $I_r\times I_s\prec B$ (if $s$ or $r$ were at least 3 then the $I_s^c$ must have been in $I_{|W|}^c$ and then complimented to become an $I_s$, and if $s=r=2$ complimenting either block would still leave you with an $I_2\times I_2$).  But $I_{|V|}\times_G I_{|W|}$ is the incidence matrix of $G$, a graph that avoids $K_{r,s}$, and hence it must avoid $I_r\times I_s$, the incidence matrix of $K_{r,s}$.  Thus we could not have had $I_r\times I_s^c\prec A$.

\end{proof}

It is known that $ex(m,K_{r,s})=\Theta(m^{2-1/s})$ for $(s-1)!\le r$, so for these values of $s$ and $r$ our bounds from Theorems \ref{T-Q3Upper} and \ref{T-Q3Lower} are sharp.  In particular, because $F_{11}=I_2\times I_2=I_2\times I_2^c$, we have the following result.

\begin{cor}\label{Q3F11}
$\forb(m,Q_3,F_{11})=\Theta(m^{3/2})$.
\end{cor}
%See Zarankiewicz problem
\section{Avoiding $1_{k,\ell}$}
In this section we study the identically 1 matrices $1_{k,\ell}$.  We first note an immediate consequence of Theorem~\ref{thm:BB}.

\begin{cor}\label{1BB}
    $\forb(m,1_{k,\ell},F)=\Theta(1)$ for $F=I_3,\ F_{10},$ or $0_{k,\ell}$.
\end{cor}
\begin{proof}
    Note that $1_{k,\ell}\prec T_{k+\ell},I^c_{k+\ell}$ and that $I_3,F_{10}\prec I_4$ and $0_{k,\ell}\prec I_{k+\ell}$.  We thus have an upper bound of $BB(k+\ell)$ by Theorem~\ref{thm:BB}.
\end{proof}

We next consider a slight generalization of a result from \cite{AK}.

\begin{thm}\label{T-ell=1}
	Let $F$ be the incidence matrix of a $(k-1)$-uniform hypergraph $\mathcal{H}$.  Then \[\forb(m,1_{k,1},F)={m\choose 0}+{m\choose 1}+\cdots +{m\choose k-2}+ex^{(k-1)}(m,\mathcal{H})\]
\end{thm}
\begin{proof}
	As a lower bound one can take all columns with fewer than $k-1$ 1's, along with the incidence matrix of a maximum $(k-1)$-uniform $\mathcal{H}$ avoiding hypergraph.  For an upper bound, note that one can have at most ${m\choose 0}+\cdots+{m\choose k-2}$ columns with fewer than $k-1$ 1's, and the columns with weight $k-1$ define the incidence matrix of a $(k-1)$-uniform hypergraph that avoids $\mathcal{H}$, and hence can be no larger than $ex^{(k-1)}(m,H)$.
\end{proof}
\begin{cor}\label{1-1F11}
	\[\forb(m,1_{k,1},I_{s_1}\times\cdots I_{s_{k-1}})={m\choose 0}+\cdots+{m\choose k-2}+ex(m,K^{(k-1)}(s_1,\ldots,s_{k-1})).\]  In particular, $\forb(m,1_{3,1},F_{11})=1+m+ex(m,K_{2,2})=\Theta(m^{3/2})$.
\end{cor}

We can get similar results when considering configurations of the form $1_{k,2}$.
\begin{thm}\label{T-ell=2}
	Let $F$ be the incidence matrix of a $k$-uniform complete $r$-partite hypergraph $\mathcal{H}$ with $r\ge k$.  Then \[\forb(m,1_{k,2},F)={m\choose 0}+{m\choose 1}+\cdots +{m\choose k-1}+ex^{(k)}(m,\mathcal{H})\]
\end{thm}
\begin{proof}
	For a lower bound, again take all columns with fewer than $k$ 1's along with the incidence matrix of a maximum $\mathcal{H}$ avoiding $k$-uniform hypergraph.  Let $A$ be a maximum matrix of $\Avoid(m,1_{k,2},F)$ and let $A'$ be a matrix obtained from $A$ by taking every column with more than $k$ 1's and removing 1's until these columns have $k$ 1's.  We claim that $A'\in \Avoid(m,1_{k,2},F)$.  Clearly $1_{k,2}\nprec A'$ (if $1_{k,2}\nprec A$ then removing 1's from $A$ can't induce this configuration) and $A'$ is simple (the columns with fewer than $k$ 1's were already distinct, and if any columns with $k$ 1's were identical we would have a $1_{k,2}$), so all that remains is to show that $F\nprec A'$.  
	
	To see this, we claim that if $F'$ is the matrix obtained by changing any 0 of $F$ to a 1 then $F'$ contains a $1_{k,2}$.  This claim is equivalent to saying that if one extends any $e\in E(\mathcal{H})$ to $e'=e\cup\set{v}$ for some $v\in V(\mathcal{H}),\ v\notin e$, then there exists an $f\in E(\mathcal{H})$ such that $|e'\cap f|=k$.  If $e$ contains no vertices that are in the same partition class as $v$, then if $f$ is any $k$-subset of $e'$ that includes $v$ then $f\in E(\mathcal{H})$ and $|e'\cap f|=k$.  If $e$ contains a vertex $v'$ that belongs to the same partition class as $v$, then $f=e'\setminus\set{v'}\in E(\mathcal{H})$ with $|e'\cap f|=k$, and thus we've proven the claim.  This means that $A$ can not contain any configuration that is obtained by taking 0's of $F$ and changing them to 1's (since $A$ avoids $1_{k,2}$), and hence the procedure of deleting 1's from $A$ can not induce an $F$  if $F\nprec A$, so we have $F\nprec A'$.

	Thus for an upper bound of $\forb(m,1_{k,2},F)$, one only needs to consider matrices where each column has at most $k$ 1's, and this clearly gives the above upper bound.
\end{proof}
\begin{cor}\label{1-2F11}
	\[\forb(m,1_{k,2},I_{s_1}\times\cdots I_{s_{k}})={m\choose 0}+\cdots+{m\choose k-1}+ex(m,K^{(k)}(s_1,\ldots,s_{k})).\]  In particular, $\forb(m,1_{2,2},F_{11})=1+m+ex(m,K_{2,2})=\Theta(m^{3/2})$.
\end{cor}

We note that in general $\forb(m,1_{k+1,1},F)\ne\forb(m,1_{k,2},F)$ when $F$ is the incidence matrix of a $k$-uniform hypergraph.  That is, the statement of Theorem~\ref{T-ell=2} can not be strengthened to include all hypergraphs as in Theorem~\ref{T-ell=1}.  For example, $Q_9$ is the incidence matrix of two disjoint edges.  It isn't difficult to see that the extremal number for this graph is $m-1$, and hence $\forb(m,1_{3,1},Q_9)=2m$.  However, the following matrix $A$ satisfies $|A|=2m+1$ and $A\in \Avoid(1,1_{2,2},Q_9)$:
\[
    A=\begin{bmatrix}
    0 & 1 & 0 & 0 \cdots 0 & 1 & 1  \cdots 1 & 0\\ 
    0 & 0 & 1 & 0 \cdots 0 & 1 & 0 \cdots 0 & 1\\ 
    0 & 0 & 0 & 1 \cdots0 & 0 & 1 \cdots 0 & 1\\ 
    \vdots & \vdots & \vdots & \vdots & \vdots &  \vdots & \vdots\\ 
    0 & 0 & 0& 0 \cdots 1 & 0 &  0 \cdots 1 & 1
    \end{bmatrix}
\]

 It should also be noted that the statement of Theorem~\ref{T-ell=2} is not as strong as possible.  For example, the theorem statement and general proof also applies to the configuration $F$ stated below, despite it not being the incidence matrix of a complete $r$-partite 3-uniform hypergraph.  It would be interesting to know of a complete characterization of $k$-uniform hypergraphs that satisfy Theorem~\ref{T-ell=2}.
\[
    F=\begin{bmatrix}
    1 & 1 & 1\\ 
    0 & 1 & 1\\ 
    1 & 0 & 1\\ 
    1 & 1 & 0
    \end{bmatrix}.
\]

Unfortunately for $1_{k,\ell}$ with $\ell>2$, this ``downgrading'' technique no longer works.  We are, however, able to obtain some partial results.

\begin{thm}\label{T-ell}
	For $\ell> 2$, \[\forb(m,1_{k,\ell},I_{s_1}\times\cdots\times I_{s_k})=\Omega(ex^{(k)}(m,K(s_1,\ldots,s_k)))\]\[ \forb(m,1_{k,\ell},I_{s_1}\times\cdots\times I_{s_k})=O(ex^{(k)}(m,K(s_1+c_1,\ldots,s_k+c_k))),\] where $c_i=(\ell-1)\max_{j\ne i}\set{\frac{s_j-1}{2}}\prod_{j\ne i}s_j$.
\end{thm}
We believe that this can be improved to $\forb(m,1_{k,\ell},I_{s_1}\times\cdots\times I_{s_k})=\Theta(ex^{(k)}(m,K(s_1,\ldots,s_k)))$, though we are unable to do so here.  Nevertheless, $ex^{(k)}(m,K(s_1+c_1,\ldots,s_k+c_k))=o(m^k)$, so this bound is non-trivial.

\begin{proof}
	The lower bound is simply the incidence matrix of the extremal hypergraph.  We first prove the upper bound for $k=2$ to demonstrate the general idea of the proof.  Let $A$ be a maximum matrix in $\Avoid(m,1_{2,\ell},I_r\times I_s)$ that has no columns with fewer than two 1's (and hence the $\forb$ function will be at most $O(m)$ larger than $|A|$). Let $C_i$ denote the set of columns of $A$ whose first 1 is in row $i$.  Note that any row $j\ne i$ restricted to $C_i$ has at most $\ell-1$ 1's (otherwise the row together with the $i$th would induce a $1_{2,\ell}$), and further note that each column of $C_i$ has a 1 in some row other than the $i$th (since every column has at least two 1's), i.e. every column of $C_i$ is identified by a 1.  We can thus use Lemma~\ref{L-avoiding rows} (after switching 0's and 1's in the lemma statement) to find  a set of rows $R_i$ such that restricted to $C_i$ these rows contain a $I_{|R_i|}$ and such that $|R_i|\ge 2^{2-\ell}|C_i|$.  We then define a bipartite graph with one vertex set corresponding to the $C_i$ column sets and the other vertex set corresponding to the rows of $A$, and we draw an edge between $C_i$ and $r$ if $r\in R_i$.  We would like to say that if this graph contains a $K_{r,s}$ (say the $r$ vertices coming from the $C_i$ vertex set and the $s$ vertices coming from the $R_i$ vertex set, which is a non-trivial assumption we will deal with later), then $A$ contains an $I_r\times I_s$.  Unfortunately, this is not true.  For example, if \[A=\begin{bmatrix}
	1 & 1 & 0 & 0\\ 
	0 & 1 & 1 & 1\\ 
	1 & 0 & 1 & 0\\ 
	0 & 1 & 0 & 1
	\end{bmatrix},\] then $A$ does not contain a $I_2\times I_2$, despite the corresponding graph being $K_{2,2}$.  The problem is that if we want to use columns from $C_i$ and $C_{i'}$ with $i<i'$, it's possible that there are 1's in the $i'$th row of $C_i$, and if these 1 columns correspond with the $I_s$ under $C_i$ then we can't actually use these columns.  Fortunately, each row below the $i$th row of $C_i$ contains fewer than $\ell$ 1's, so this problem can't happen too many times.  We claim that if instead of having an $I_s$ simultaneously under $r$ different $C_i$ we had an $I_{s+c_2}$, where $c_2=(\ell-1)\frac{r(r-1)}{2}$, simultaneously under $r$ different $C_i$, then we could find an $I_r\times I_s$.

	 Assume that we have this situation with the $i$'s of our $C_i$'s belonging to the set $\set{i_1,\ldots,i_r}_<$, and let $R'_0$ denotes the set of rows that contain the simultaneous $I_{s+c_2}$ under these $C_i$, noting that $|R'_0|=s+(\ell-1)\frac{r(r-1)}{2}$.  For $r\in R'_0$, we will say that its corresponding column restricted to $C_{i_j}$ is the column where $r$ contains the 1 it contributes to the $I_{|R'_0|}$ in $C_{i_j}$.  Note that restricted to the $r-1$ rows $\set{i_2,\ldots,i_r}$, $C_{i_1}$ contains at most $(\ell-1)(r-1)$ 1's (as each row has at most $\ell-1$ 1's).  Thus if $B_1$ is the set of columns of $C_{i_1}$ with 1's in these rows we have $|B_1|\le (\ell-1)(r-1)$.  Define $R'_1\sub R'_0$ to be the set of rows that have corresponding columns in $C_{i_1}$ that are not in $B_1$, and hence $|R'_1|\ge |R'_0|-(\ell-1)(r-1)=s+(\ell-1)\frac{(r-1)(r-2)}{2}$.  Note that restricted to the corresponding columns of $R'_1$ and the rows $\set{i_2,\ldots,i_r}$, $C_{i_1}$ is identically 0.  We can similarly define the subset $R'_2\sub R'_1$ consisting of the rows whose corresponding columns in $C_{i_2}$ are 0 in the rows $\set{i_3,\ldots,i_r}$ (row $i_1$ is automatically identically 0 restricted to $C_{i_2}$ since $i_1<i_2$) with $|R'_2|\ge |R'_1|-(\ell-1)(r-2)\ge s+ (\ell-1)\frac{(r-2)(r-3)}{2}$.  We repeat this process until we reach the set $R'_r$ which satisfies $|R'_r|\ge s$ and under each $C_{i_j}$, the corresponding columns of $R'_r$ are identically 0 in the other $i_{j'}$ rows.  This gives an $I_r\times I_s$.  
	 
	 However, to guarantee an $I_r\times I_s$ in $A$ it is insufficient to simply guarantee the existence of a $K_{r,s+c_2}$ in the graph we constructed, since we could have the $s+c_2$ vertices coming from the $C_i$ vertex set instead of the row vertex set. To remedy this, we must increase $r$ by a suitable amount as well, namely by $c_1=(\ell-1)\frac{s(s-1)}{2}$, as in this case a symmetric argument will guarantee our result.  Thus the existence of a $K_{r+c_1,s+c_2}$ in this graph guarantees an $I_r\times I_s$, so the graph must have $O(ex(m,K_{r+c_1,s+c_2}))$ edges, and hence $|A|=O(ex(m,K_{r+c_1,s+c_2}))$ as well.
	
	For the general problem, again consider a maximum $A$ with every column having at least $k$ 1's and define the set $C(i_1,\ldots,i_{k-1})$ to be the columns which have their first $k-1$ 1's in rows $i_1,\ldots, i_{k-1}$ and with $i_j>i_{j-1}$.  Again we can find rows $R(i_1,\ldots,i_{k-1})$ such that the number of rows is proportional to the number of columns of $C(i_1,\ldots,i_{k-1})$, and restricted to these rows and columns there is a large identity matrix.  We can then define a $k$-uniform $k$-partite hypergraph with vertex sets $V_j$ for $1\le j<k$ corresponding to all possible choices of $i_j$, and vertex set $V_k$ corresponding to all rows of $A$.   We then add the hyperedge $\set{i_1,\ldots,i_{k-1},r}$ to our hypergraph iff $r\in R(i_1,\ldots,i_{k-1})$.  If this hypergraph contains a $K^{(k)}(s_1+c_1,\ldots,s_k+c_k)$ where $c_i=(\ell-1)\max_{j\ne i}\set{\frac{s_j-1}{2}}\prod_{j\ne i}s_j$, then we claim that $A$ contains an $I_{s_1}\times\cdots\times I_{s_k}$.
	
	Assume that this hypergraph contains a $K^{(k)}(s_1+c_1,\ldots,s_k+c_k)$, say on the vertex sets $V'_1,\ldots,V'_k$ with $V'_j\sub V_j$ and $|V'_i|=s_i+c_i$ (again, an assumption we'll have to address later).
	First note that if $i_j\in V'_j$ and $i_{j'}\in V'_{j'}$ with $j<j'$, then $i_j<i_{j'}$.  Indeed, because we have a complete $k$-partite hypergraph, $i_j\in V'_j$ and $i_{j'}\in V'_{j'}$ means that there exists an edge containing both $i_j$ and $i_{j'}$ from these vertex sets.  If $j'<k$ then this edge corresponds to a column whose $j$th 1 is in row $i_j$ and $j'$th 1 is in row $i_{j'}$,and if $j<j'$ this only makes sense if $i_j<i_{j'}$.  If $j'=k$ then the $i_{j'}$th row must come after the rows where this column has its first $k-1$ 1's by definition, and hence again $i_j<i_{j'}$.  This means that for any $C(i_1,\ldots,i_{k-1}),\ i\in V_j$ with $i\ne i_j$ and $j<k-1$, the $i$th row of $C(i_1,\ldots,i_{k-1})$ is identically 0 (since its $(j+1)$th row with a 1 in it comes from row $i_{j+1}>i$ and its $(j-1)$th comes from $i_{j-1}<i$ if $j\ne1$), and hence when choosing corresponding rows from $V'_k$ the only potential pitfall will be the rows from $V'_{k-1}$ (as it is possible for $C(i_1,\ldots,i_{k-1})$ to have 1's in row $i\ne i_{k-1}$ even if $i\in V'_{k-1}$).
	
	For $j<k$ let $V''_j\sub V'_j$ be any subset with $|V''_j|=s_j$ and let $R'_0$  be the set of rows corresponding to the $I_{s_k+c_k}$ simultaneously under all of the $C(i_1,\ldots,i_{k-1})$ columns with $i_j\in V''_j$, and we emphasize that our observations in the preceding paragraph shows us that the rows of $R'_0$ lie entirely below the rows of every $V_j''$ for $1\le j<k-1$.  Let $i_1,\ldots,i_{k-2}$ be any fixed elements from the $V''_j$'s.  Restricted to the columns $C(i_1,\ldots,i_{k-1})$, where $i_{k-1}$ varies amongst all $V''_{k-1}$, we perform the same procedure that we used for the $k=2$ case to obtain a set of rows $R'_1$, after removing at most $(\ell-1)\frac{s_{k-1}(s_{k-1}-1)}{2}$ rows from $R'_0$, such that that for any $i_{k-1}\in V''_{k-1}$ and any corresponding column of $R'_1$ restricted to the rows $V_{k-1}''\setminus\set{i_{k-1}}$,  $C(i_1,\ldots,i_{k-1})$ is identically 0.  We then repeat this process for all possible sequences of $i_1,\ldots,i_{k-2}$, in total removing at most $\frac{s_{k-1}(s_{k-1}-1)}{2}\prod_{j<k-1}s_j$ rows (which in the worst case scenario is $(\ell-1)\max_{j\ne k}\set{\frac{s_j-1}{2}}\prod_{j\ne k}s_j$).  In the end we are left with a set $R'\sub R'_0$ with $|R'|\ge s_k$ and in the corresponding columns of any $C(i_1,\ldots,i_{k-1})$ for $i_j\in V''_j$ and restricted to the rows $V''_{k-1}\setminus\set{i_{k-1}}$ the matrix is identically 0.  This gives an $I_{s_1}\times\cdots\times I_{s_k}$ in $A$.  Hence the hypergraph can have at most $ex^{(k)}(m,K^{(k)}(s_1+c_1,\ldots,s_k+c_k))$ edges, which means that overall $|A|=O(ex^{(k)}(m,K^{(k)}(s_1+c_1,\ldots,s_k+c_k)))$.
\end{proof}
Next we consider $\forb(m,1_{k,1},F_{11})$ The following was proven by Gy\'arf\'as et. al. \cite{gyfs}.
\begin{prop}\label{Pr 1F11}
    $\forb(m,1_{4,1},F_{11})=\Theta(m^{3/2})$.
\end{prop}
Proposition~\ref{Pr 1F11} is a corollary of the following theorem that was first proven by F\"uredi and Sali \cite{fusa}
\begin{thm}\label{thm:1kkrs}
$r\ge s\ge k-2\ge 1$ be fixed integers. Then $\forb(m,1_{k,1},I_r\times I_s)=O(m^{k-1-\frac{1}{s}\binom{k-1}{2}})$. Furthermore, if $r\ge (s-1)!+1$ and $s\ge 2k-4$, then
$\forb(m,1_{k,1},I_r\times I_s)=\Theta(m^{k-1-\frac{1}{s}\binom{k-1}{2}})$
\end{thm}
For the sake of completeness we give a simpler proof extending ideas of \cite{gyfs}
\iffalse
For the proof we will apply the kernel ($\Delta$-system) method of F\"uredi \cite{furedi-kernel}. Let $\mathcal F$ be a $k$-uniform set system on $\{1,2,\ldots m\}$, furthermore let $\mathcal{M}(F,\mathcal{F})=\{F\cap F'\colon F\ne F'\in\mathcal{F}\}$. For a $k$-partite $k$-uniform system with partite classes $V_1,V_2,\ldots ,V_k$ a projection $\pi\colon V_1\cup V_2\cup\ldots\cup V_k\rightarrow\{1,2,\ldots ,\}$ is defined by $\pi(x)=i\iff x\in V_i$. For a subset $A\subseteq V_1\cup V_2\cup\ldots\cup V_k$ define $\pi(A)=\{\pi(a)\colon a\in A\}$. Let us recall that a $t$-star with kernel $X$ is a collection of $t$ sets $F_1,F_2,\ldots F_t$, such that $F_i\cap F_j=X$ for all $1\le i<j\le t$. F\"uredi proved the following.
\begin{thm}\label{thm:furedi-kernel}
For any positive integers $k<t$, there exists a positive real number
$c = c(k, t)$ with the following property: If $\mathcal F$ is a $k$-uniform hypergraph, then we can choose a
subsystem $\mathcal{F}^*\subset \mathcal F$ such that:
\begin{enumerate}
\item  $|\mathcal{F}^*|>c|\mathcal{F}|$.
\item Every pairwise intersection in $\mathcal{F}^*$ is a kernel of a $t$-star of $\mathcal{F}^*$.
\item $\mathcal{F}^*$ is $k$-partite with partite classes $V_1,V_2,\ldots ,V_k$.
\item There exists a set system $\mathcal{M}$ on $\{1,2,\ldots ,k\}$ such that $\pi(\mathcal{M}(F,\mathcal{F}^*))=\mathcal{M}$ for all $F\in\mathcal{F}^*$ Furthermore, $\mathcal{M}$ is closed under intersection. 
\end{enumerate}
\end{thm}
\fi
We need the following theorem of Alon and Shikhelman.  Let $ex(m,G,H)$ mean the largest possible number of subgraphs isomorphic to $G$ in an $m$-vertex graph that does not have $H$ as subgraph. Alon and Shikhelman prove 
\begin{thm}[Alon and Shikhelman]\label{thm:ash}
Let $r\ge s\ge k-1$ be fixed integers. Then $ex(m,K_k,K_{r,s})=O(m^{k-\frac{1}{s}\binom{k}{2}})$, furthermore, if $r\ge (s-1)!+1$ and $s\ge 2k-2$, then $ex(m,K_k,K_{r,s})=\Theta(m^{k-\frac{1}{s}\binom{k}{2}})$.
\end{thm}
\begin{proof}[Simpler Proof of Theorem~\ref{thm:1kkrs}]
Let $A\in\Avoid(m,1_{k+1,1},I_r\times I_s)$. We can inductively conclude that $\forb(m,1_{k+1,1},I_r\times I_s)^{< k}=O(m^{k-1-\frac{1}{s}\binom{k-1}{2}})$, base case being $k=3$. 
let $A'$ be obtained by deleting columns of sum less than $k$ from $A$. 
Consider columns of $A'$ as characteristic vectors of a $k$-uniform hypergraph $\mathcal{F}$. Let $\mathcal{F}'_1$ be a largest size $k$-partite subhypergraph of $\mathcal{F}$, with partite classes $V_1,V_2,\ldots,V_k$.  It is well know that $|\mathcal{F}|\le c_k|\mathcal{F}'_1|$ for some constant $c_k$. Let $\mathcal{H}_i$ be the $(k-1)$-partite graph induced by $\mathcal{F}'_1$ after ignoring $V_i$. Observe that no $\mathcal{H}_i$ contains $K_{r,s}$ as a trace. Call a  hyperedge $F\in \mathcal{F'}_1$ 1-thick if restricted to each $\mathcal{H}_i$, $F$ is contained in at least $r+s-2$ other hyperedges of $\mathcal{F'}_1$, and call $F$ 0-thick otherwise.  There are at most $(r+s-2)|E(\mathcal{H}_i)|$ 0-thick edges. Recursively define $\mathcal{F}'_i$ to consist of all $F\in\mathcal{F}'_{i-1}$ that are $i-1$ thick, and call $F\in \mathcal{F}'_i$  $i$-thick if restricted to each $\mathcal{H}_i$ it is contained in at least $r+s-1$ hyperedges of $\mathcal{F}'_i$.  By the same reasoning as before, $|\set{F|F\in \mathcal{F}'_{i-1},F\notin\mathcal{F}'_i}|\le(r+s-2)|E(\mathcal{H}_i)|$, and thus the number of $F\in \mathcal{F}'_1$ that are not $k$-thick is at most $k(r+s-2)|E(\mathcal{H}_i)|=O(m^{k-1-\frac{1}{s}\binom{k-1}{2}})$ by the inductive hypothesis.  On the other hand, the 2-shadow of $\mathcal{F}'_{k}$ can not contain an $K_{r,s}$.  

Assume in contrary that this is the case and consider an edge $\set{x_1,x_2}$ used in this $K_{r,s}$ and let $F_0$ be a $k$-thick edge with $\set{x_1,x_2}\in F_0$.  If $F_0$ contains no vertex in $(V(K_{r,s})\setminus \set{x_1,x_2})\cap V_1$, then define $F_1=F_0$.  Otherwise, by definition of $F_0$ being a $k$-thick edge there exists $r+s-1$ hyperedges that are $(k-1)$-thick and that differ with $F_0$ only in the vertex set $V_1$.  By the pigeonhole principle, one of these hyperedges, call it $F_1$, does not contain any vertex of $(V(K_{r,s})\setminus \set{x_1,x_2})\cap V_1$ and still has $\set{x_1,x_2}\in F_1$.  Continue this way, defining $F_{i}$ to be a $(k-i)$-thick hyperedge that contains $\set{x_1,x_2}$ and no vertices of $(V(K_{r,s})\setminus \set{x_1,x_2})\cap \bigcup_{j\le i} V_j$, and we can do this at each step by the way we defined $(k-i)$-thickness.  In the end we obtain a hyperedge $F_k$ that contains $\set{x_1,x_2}$ and no other vertices of the $K_{r,s}$.  We can repeat this process for each edge of the $K_{r,s}$, and thus these hyperedges contain $I_r\times I_s$ as a trace. 
Thus, we inferred that the 2-shadow does not have $K_{r,s}$ as a subgraph. Apply Theorem~\ref{thm:ash} to the graph determined by the 2-shadow of $\mathcal{F}'_{k}$ and obtain that the number of $K_k$ subgraphs is at most $O(m^{k-\frac{1}{s}\binom{k}{2}})$, which clearly is an upper bound for $|\mathcal{F}'_{k}|$. \par
Summarising,
\[
    |A|=|A\setminus A'|+|A'|\le |A\setminus A'|+\frac{1}{c_k}(k(r+s-1)|E(\mathcal{H}_i)|+|\mathcal{F}'_{k}|)=O(m^{k-\frac{1}{s}\binom{k}{2}}).
\]
To prove the lower bound take a graph $G$ that gives the lower bound in Alon-Shikhelman' Theorem and let $\mathcal F$ consists of those $k$-subsets of the vertices that induce a complete graph. Since $G$ does not have $K_{r,s}$  subgraph, $\mathcal F$ does not have $K_{r,s}$ as \emph{trace}, so if $A$ is the vertex-edge incidence matrix of $\mathcal F$, then $A\in\Avoid(m,1_{k+1,1},I_r\times I_s)$.
\end{proof}

Note that the upper bound in Proposition~\ref{Pr 1F11} is obtained by putting $r=s=k-1=2$. The lower bound in Theorem~\ref{thm:1kkrs} does not give the lower bound of Proposition~\ref{Pr 1F11} directly, however the vertex-edge incidence matrix of a maximal $C_4$-free grah works.

\begin{rem}\label{rem:no-conj}
Despite the largest product avoiding $1_4$ and $I_r\times I_s$ being a 1-fold product, Theorem~\ref{thm:1kkrs} shows that one can make $\forb(m,1_4,I_r\times I_s)=\Theta(m^{3-\epsilon})$.  Thus the best we could hope for as an extension of Conjecture~\ref{grand} for general forbidden families is $\forb(m,F,G)=o(m^p)$ if $\forb(m,F)=\Theta(m^p)$ and there exists no $p$-fold product avoiding both $F$ and $G$. However, we do not dare to formulate this as a conjecture. 
\end{rem}
The following extension of 
Proposition~\ref{Pr 1F11} was proven in \cite{fusa}.
\begin{prop}\label{prop:1kf11}
 Let $k\ge 3$ be a positive integer. Then 
  $\forb(m,1_{k,1},F_{11})=\Theta(m^{3/2})$.
\end{prop}
An alternate proof of this Proposition could be given using similar ideas as in the simpler proof of Theorem~\ref{thm:1kkrs}.
\section{Avoiding $F_9$}
\[
    F_9=\begin{bmatrix}
    1 & 0 & 0\\ 
    0 & 1 & 0\\ 
    0 & 0 & 1\\ 
    0 & 0 & 1
    \end{bmatrix}
\]

\begin{thm}\label{F9}
	$\forb(m,Q_3(t),F_9)=\Theta(m)$.
\end{thm}
\begin{proof}
	Note that $I_m$ gives the lower bound.  For the upper bound, we first take a look at what our preliminary data tells us.  We have that $F_9\prec I_3\times I_2^c$, so by Theorem~\ref{T-Q3Upper} we know that $\forb(m,Q_3(t),F_9)=O(m^{3/2})$.  It also isn't too hard to show (using methods similar to what we'll use below) that $|\tilde{A}|=O(m)$ if $\tilde{A}\in \Avoid(m,Q_3(t),F_9)$ meets all the requirements of the $A'_j$ matrices in the statement of Theorem~\ref{T-Q3Main}, so we have $\forb(m,Q_3(t),F_9)=O(m\log m)$ by Corollary~\ref{C-log bound}, and this suggests that $\forb(m,Q_3(t),F_9)=O(m)$.  Unfortunately, this is as far as we can get using the results of Theorem~\ref{T-Q3Main}.  However, by following the same basic argument of the proof of the theorem, and by using the extra information that we must also avoid $F_9$, we will be able to show the $O(m)$ result.
	
	Let $A\in \Avoid(m,Q_3(t),F_9)$ such that $|A|$ is maximal and assume $|A|=\omega(m)$.  Let $k$ be the largest integer such that $t\cdot I_k\prec A$ (we don't consider the $R_1$ rows as that technical step will not be required for this proof).  Rearrange rows so that this $t\cdot I_k$ appears in the first $k$ rows and let $C_i$ denote the set of columns with a 1 in row $i$ and $C^2$ the columns with no 1's in the first $k$ rows (and we can assume that $k\ge 3$, thus having no $Q_3(t)$ implies that no column can have two 1's in the first $k$ rows, so all columns belong to precisely one of these sets).
	
	\begin{lem}\label{L-nzero}
		No row $r$ restricted to $\bigcup C_i$ is identically 0.	
	\end{lem}
	\begin{proof}
		Assume there is an $r$ such that $r$ is identically 0 restricted to $\bigcup C_i$.  Consider how many 1's $r$ has in $C^2$.  If $r$ has fewer than $t$ 1's, then by using the standard induction with row $r$ we see that $|C_r|\le t-1=O(1)$, so we could inductively conclude that $|A|=O(m)$.  Otherwise there are at least $t$ 1's, in which case one could use this row to find a $t\cdot I_{k+1}$ in $A$, a contradiction.
	\end{proof}
	\begin{lem}\label{L-one0}
		If row  $r$ with $r>k$ has a 0 restricted to $\bigcup C_i$ then it has 0's in precisely one $C_i$.	
	\end{lem}
	\begin{proof}
		Assume $r$ has a 0 in $C_{i}$ and $C_{i'}$.  If there is a 1 in any column of $C_{i''},\ i''\ne i,i'$, then by taking these columns and rows $r,\ i,\ i',\ $and $i''$ we get an $F_{9}$.  If every $C_{i''}$ is identically 0 then by Lemma \ref{L-nzero} one of $C_i,\ C_{i'}$ must have a 1 in some column, say $c\in C_i$.  But then by taking $c$, the column with a 0 in $C_{i'}$, and any column in any other $C_{i''}$ along with the relevant rows gives an $F_9$.
	\end{proof}
	\begin{lem}
		$|C^2|=O(m)$.	
	\end{lem}
	\begin{proof}
		Assume $|C^2|=\omega(m)$, in which case there must exist a $Q_3(t;0)$ in $C^2$ and it must lie below the top $k$ rows.  But as $k\ge 3$, for any two rows $r_1,r_2\ge k$ one can find a $\begin{bmatrix}
		1\\ 1
		\end{bmatrix}$ in some $C_i$ (if $r_1$ has 0's in $C_1$ and $r_2$ has 0's in $C_2$ then neither can have 0's in $C_3$ by Lemma \ref{L-one0}).  Thus whatever rows the $Q_3(t;0)$ lies in one can find a column to give a $Q_3(t)$, a contradiction.
	\end{proof}
	\begin{lem}
		$|\bigcup C_i|=O(m)$.	
	\end{lem}
	\begin{proof}
		Let $R_i$ denote $C_i$ restricted to its rows that are not identically 1.  Note that $R_i$ is a simple matrix, and let $r_i$ denote the number of rows it has.  We can't have $|C_i|>c_{3,t} r_i$ (as then we could find a $Q_3(t;0)$ in $R_i$ and take any column of $C_{i'},\ i'\ne i$ to get a $Q_3(t)$), so we must have $|\bigcup C_i|=\sum |C_i|\le c_{3,t} r_i\le c_{3,t} m=O(m)$.
	\end{proof}
	Thus $|A|=|\bigcup C_i|+|C^2|=O(m)$.
\end{proof}

\begin{thm}
    $\forb(m,1_{k,\ell},F_9)=\Theta(m)$ provided we don't have $k=\ell=1$.
\end{thm}
\begin{proof}
    Note that $I_m$ gives the lower bound.  Let $A$ be a maximum sized matrix in $\Avoid(m,1_{k,\ell},F_9)$ and apply the standard induction on any row $r$ to get the matrix of repeated columns $C_r$.  If $C_r\le BB(k+\ell+1)$ then we inductively conclude that $|A|=O(m)$.  Otherwise, we must have either a $I_{3},I_{k+\ell+1}^c$ or $T_{k+\ell+1}$ in $C_r$.  As $1_{k,\ell}\prec I_{k+\ell+1}^c,T_{k+\ell+1}$, we must have $I_3\prec C_r$ and hence $[01]\times I_3\prec A$.  But $F_9\prec [01]\times I_3$, which contradicts $F_9\nprec A$.
\end{proof}

It is possible to get a finer value for $\forb(m,1_{k,\ell},F_9)$, and even an exact value in a few select cases when $m$ is sufficiently large.  We say that a column in $A$ is an $n$-column if its column sum is $n$.  We define $\Avoid(m,F)^{=n}$ to be the set of matrices $A$ that avoid $F$ and whose columns are all $n$-columns, and analogously we define $\forb(m,F)^{=n}$.  We similarly define $\Avoid(m,F)^{\ge n}$ and $\forb(m,F)^{\ge n}$.  For columns $c,d$ we will let $c\cap d$ denote the set of rows that $c$ and $d$ both have 1's in, and we similarly define $c\cup d$.  

\begin{lem}\label{L-specific small}
	For any fixed $t>k$, $\forb(m,1_{k,\ell},F_9)^{=t}\le (BB(k+2)+\ell)2^t$.
\end{lem}
\begin{proof}
	We first consider the $\ell=2$ case (the $\ell=1$ case is trivial).  Assume the first column $c$ of a matrix $A\in\Avoid(m,1_{k,2},F_9)^{=t}$ has all its 1's in the first $t$ rows.  For $S\sub[t]$ with $|S|\le k-1$, let $C_S$ denote the set of columns $c'$ of $A$ such that $c\cap c'=S$, and note that every column of $A$ belongs to precisely one such set.  But note that $|[t]\setminus S|\ge 2$, which means that for every $S$ there exists two rows such that $c$ has a 1 in these rows and every column of $C_S$ has 0's.  Hence, below the first $t$ rows the columns of $C_S$ can not induce an $I_2$ (as in these rows $c$ is 0, so these together with the 2 rows mentioned above give an $F_9$).  But $C_S$ is a simple matrix so if $|C_S|> BB(k+2)$ it must contain a $T_{k+2}$, which in particular contains $1_{k,2}$.  Thus $|C_S|\le BB(k+2)$ for all $S$, and as there are fewer than $2^t$ such sets (and they partition all of $A$), we must have $|A|\le BB(k+2)2^t$.
	
	For $\ell>2$ one can consider $S\sub [t]$ with $|S|\ge k$, but for such $S$ we must have $|C_S|<\ell$ to avoid $1_{k,\ell}$, so we have the bound $|A|\le (BB(k+2)+\ell)2^t$.
\end{proof}
\begin{lem}\label{L-large}
	$\forb(m,1_{k,\ell},F_9)^{\ge c_{k,\ell}}=c'_{k,\ell}$ where $c_{k,\ell}=2^{\ell-1}(k+1)-1$ and $c'_{k,\ell}=O(1)$.
\end{lem}
\begin{proof}
	We have $c_{k,1}=k$, so the statement is trivially true for $\ell=1$.  Assume for the purpose of induction that this result is true up to $\ell-1$ and consider a matrix $A\in \Avoid(m,1_{k,\ell},F_9)^{\ge c_{k,\ell}}$ and any column $d$ in $A$.  Let $R_0$ denote the rows where $d$ has 0's and $R_1$ the rows where $d$ has 1's.  We claim that restricted to $R_0$ there exists no $I_{z}$ where $z=(\ell-1)(c'_{k,\ell-1}+1)+1$.  Indeed, any two columns of such a $I_{z}$, say $c_1$ and $c_2$, induce an $I_2$ in $R_0$, and using column $d$ as well as $c_1$ and $c_2$ would give a $\begin{bmatrix}
	0 & 1 & 0\\ 
	0 & 0 & 1
	\end{bmatrix}$, thus if there exists two rows in $R_1$ where $c_1$ and $c_2$ are both 0 then one could find an $F_9$.  As $d$ has at least $2^{\ell-1}(k+1)-1$ 1's, we must have (restricted to $R_1$) $|c_1\cup c_2|\ge 2^{\ell-1}(k+1)-2$ (otherwise there will be at least two rows of $R_1$ that aren't covered by $c_1$ and $c_2$), and hence one of these $c_i$ must have at least $2^{\ell-2}(k+1)-1=c_{k,\ell-1}$ 1's in $R_1$.  Thus all but at most one of the $I_{c}$ columns must have at least $c_{k,\ell-1}$ 1's in $R_1$.  Let $A'$ be $A$ restricted to the $R_1$ rows and the columns of the $I_{c}$ that have at least $c_{k,\ell-1}$ 1's in these rows.  $A'$  need not be simple, but each column can be repeated at most $\ell-1$ times before inducing a $1_{k,\ell}$, so there are at least $c'_{k,\ell-1}+1$ distinct columns in $A'$.  But by the inductive hypothesis this means that there exists either an $F_9$ (in which case we're done) or a $1_{k,\ell-1}$ in $R_1$, and using column $d$ in addition to this would give a $1_{k,\ell}$.  Thus there can exist no $I_{c}$ in $R_0$, but similarly there can't exist sufficiently large $I^c$'s or $T$'s (as these automatically contain $1_{k,\ell}$), so restricted to $R_0$ there can be at most $BB(c)$ column types.  
	
	Any column type restricted to $R_0$ with at least $k$ 1's  can't appear more than $\ell-1$ times (as this would give a $1_{k,\ell}$), and columns restricted to $R_0$ with fewer than $k$ 1's must have at least $c_{k,\ell}-(k-1)=2^{\ell-1}(k+1)-1-(k-1)\ge 2^{\ell-2}(k+1)-1=c_{k,\ell-1}$ 1's in $R_1$ (since every column of $A$ has at least $c_{k,\ell}$ 1's), and thus can't appear more than $c'_{k,\ell-1}$ times without inducing in $R_1$ either an $F_9$ or a $1_{k,\ell-1}$ (and hence a $1_{k,\ell}$ by using column $d$).  Thus each of the constant number of column types appears at most a constant number of times, so we have $\forb(m,1_{k,\ell},F_9)^{\ge c_{k,\ell}}\le BB(c)(\ell-1+c'_{k,\ell-1})=O(1)$.
\end{proof}
\begin{lem}\label{L-ideal}
	For any fixed $t$, if $A\in \Avoid(m,1_{k,\ell},F_9)^{=t}$ and if $c$ is any column of $A$, then there are at most $O(1)$ columns $c'$ of $A$ with $|c\cap c'|<t-1$.
\end{lem}
\begin{proof}
	The statement is trivially true for $t>k$ (since there can only be at most $O(1)$ such columns by Lemma~\ref{L-specific small}) and $t=1$, so assume $1< t\le k$.  Rearrange rows so that the 1's of $c$ appear in the first $t$ rows of $A$, and for any $S\sub [t]$ let $C_S$ denote the columns of $A$ with $c\cap  c'=S$.  If $S$ is a set with $|S|<t-1$, then as argued in Lemma \ref{L-specific small} the columns of $C_S$ can't contain an $I_2$ (since there exists at least two of the first $t$ rows with 1's in $c$ and 0's in all of $C_S$) and it also can't contain a $T_{k+\ell+1}$, so we must have $|C_S|\le BB(k+\ell+1)$, and since there are fewer than $2^t$ such sets of $A$ we have $|A|\le BB(k+\ell+1)2^t=O(1)$.
\end{proof}
Let $A^{\ne t}$ denote the collection of columns of a matrix $A$ that are not $t$-columns.
\begin{lem}\label{L-only one}
	There exists a constant $p\in\mathbb{N}$ such that if $A\in \Avoid(m,1_{k,\ell},F_9)$ with $|A|\ge 2p c_{k,\ell}+c'_{k,\ell}$, then there exists a unique $t\le k$ such that $|A^{\ne t}|\le (2p-1)k+p$.
	 Further, there exists $t-1$ rows where every $t$-column of $A$ has $t-1$ 1's in these rows.
\end{lem}
Note that implicitly this statement requires that $m$ be sufficiently large in order for $|A|\ge 2p c_{k,\ell}+c'_{k,\ell}$.

\begin{proof}
    	Let $p$ be the smallest (constant) value such that it is larger than $c_{k,\ell}+1$, $c'_{k,\ell}$ and all the $O(1)$ constants obtained from Lemma \ref{L-specific small} for $k<t\le c_{k,\ell}$ and Lemma \ref{L-ideal} for $t\le k$.  Let $t\le k$ be the smallest $t$ such that $A$ contains at least $2p$ $t$-columns (and at least one such $t$ must exist by the previous lemmas and the assumption that $|A|\ge 2p c_{k,\ell}+c'_{k,\ell}$).  We claim that this is the only such $t$.  Indeed, by Lemma \ref{L-ideal} at most $p$  of these $t$-columns don't intersect in the same $t-1$ rows, or in other words, at least $p$ of these $t$-columns must intersect in the same $t-1$ rows, say the first $t-1$.  Their last 1's must all be in separate rows, and this induces an $I_p$ below the first $t-1$ rows.  We claim that $A$ contains no $t'$-column with $t<t'<p-1$.  Indeed, such a $t'$ must contain at least two 1's outside of the first $t-1$ rows (since $t'>t$), and it does not have 1's in at least two rows of the $I_p$ (since $t'<p-1$).  Take two rows where $t'$ has 1's below the first $t-1$ rows and two rows where $t'$ does not have 1's in rows of the $I_p$, as well as the $t'$ column and the two columns of the $I_p$ that give an $I_2$ from the rows chosen.  The $t'$ column gives a $\begin{bmatrix}
	0 \\ 0\\ 1\\ 1
	\end{bmatrix}$ (the first two rows where it doesn't intersect with $I_p$) and the other columns give a $\begin{bmatrix}
	1 & 0\\ 
	0 & 1\\ 
	0 & 0\\ 
	0 & 0\\
	\end{bmatrix}$ (since all these rows are after the first $t-1$, and hence every column of the $I_p$ has only one 1 in these columns), and this gives an $F_9$, so there can be no such $t'$-columns (the same argument shows that any $t$-column must have 1's in the first $t-1$ rows).  As $t$ was chosen to be the smallest column type with at least $2p$ columns, in addition to the fact that $\forb(m,1_{k,2},F_9)^{\ge p}\le c'_{k,\ell}\le p$, it is the only such column type with at least this many columns, and thus $A$ can contain at most $(2p-1)t+p\le (2p-1)k+p$ columns that are not $t$-columns. 
\end{proof}

\begin{cor}
	For $m$ sufficiently large, $\forb(m,1_{k,1},F_9)=m+c_k$, where $c_k$ is some constant depending only on $k$.
\end{cor}
\begin{proof}
	Note that $I_m$ gives the lower bound.  For any $A\in \Avoid(m,1_{k,1},F_9)$ with $|A|\ge 2p c_{k,\ell}+c'_{k,\ell}$ and $m$ sufficiently large, Lemma \ref{L-only one} tells us that only one column type appears more than $2p$ times, say the $t$-columns for some $t\le k$.  But $|A^{=t}|\le m-t+1$ (only this many $t$-columns can intersect in the same $t-1$ rows, and every $t$-column in $A$ does this) and $|A^{\ne t}|\le (2p-1)k+p$, and hence $|A|\le m-t+1+(2p-1)k+p\le m+(2p-1)k+p,$ where $(2p-1)k+p$ is a constant depending only on $k$.
\end{proof}
\begin{cor}\label{F9-ell}
	For $\ell\ge 2$ and $m$ sufficiently large, \[\forb(m,1_{k,\ell},F_9)=\forb(m,1_{k+1,1},F_9)+\ell-1=m+c_{k+1}+\ell-1.\]
\end{cor}

\begin{proof}
    Let $p$ be the constant defined in Lemma~\ref{L-only one} and let
     $A\in \Avoid(m,1_{k,\ell},F_9)$ with $|A|\ge 2p c_{k,\ell}+c'_{k,\ell}$. We claim that $A$ contains at most $\ell-1$ columns with at least $k$ 1's.  Indeed, consider the $I_p$ in $A$ and note that any column with at least $k$ 1's must have 1's in all but at most one of the rows that contains the $I_p$ (as otherwise one can find an $F_9$).  As $p>k+\ell$, there can exist at most $\ell-1$ such columns before the columns induce a $1_{k,\ell}$.  Thus we can reduce sufficiently large $A\in \Avoid(m,1_{k,\ell},F_9)$ to an $A'\in \Avoid(m,1_{k+1,1})$ after removing at most $\ell-1$ columns, so we have $\forb(m,1_{k,\ell},F_9)\le \forb(m,1_{k+1,1},F_9)+\ell-1$.
	
	Take any $A\in \forb(m,1_{k+1,1},F_9)$ and let $A'$ be $A$ after adjoining $\ell-1$ $(m-1)$-columns to $A$.  $A'$ avoids $F_9$ (since $A$ avoided $F_9$ and no $(m-1)$-column can contain an $F_9$ since they don't have two 0's) and it avoids $1_{k,\ell}$ (as there are only $\ell-1$ columns of $A'$ with at least $k$ 1's).  Hence $A'\in \Avoid(m,1_{k,\ell},F_9)$ so we have 	$\forb(m,1_{k,\ell},F_9)\ge\forb(m,1_{k+1,1},F_9)+\ell-1$.
\end{proof}
It is somewhat surprising that, despite the extra care needed to deal with $\ell>1$ in our lemmas, the value of $\ell$ only contributes linearly to $\forb(m,1_{k,\ell},F_9)$.  This will also be the case for $\forb(m,1_{k,\ell},Q_9)$ in the next section, and this provides some evidence that the upper bound for $\forb(m,1_{k,\ell},I_{s_1}\times\cdots I_{s_k})$ should asymptotically be the same as $\forb(m,1_{k,2},I_{s_1}\times\cdots I_{s_k})$.

The exact value of $c_k$ seems to be difficult to compute in general, but for specific (small) values of $k$ it is possible to compute.

\begin{prop}
	$c_2=1$.
\end{prop}
\begin{proof}
	Take $[0_{m,1}|I_m]$.  Clearly this avoids $F_9$ and this includes every column that avoids $1_{2,1}$.
\end{proof}

\begin{prop}
	$c_3=2$.
\end{prop}
\begin{proof}
	To do better than our bound of $c_2$ we must use $2$-columns in our construction (and hence we must use $\Theta(m)$ of them all intersecting in some row, say row 1).  In such a construction, there can't be more than two 1-columns (otherwise we'd have an $I_2$ below row 1, and then taking any 2-column that doesn't intersect with these 1-columns gives an $F_9$) and we can only have one 0-column.  Thus we must have $\forb(m,1_{3,1},F_9)\le 1+2+(m-1)=m+2$, and this can be achieved by considering $A$ with the 0-column, two 1-columns in rows 1 and 2, and all 2-columns that have 1's in row 1.
\end{proof}

\begin{prop}\label{P-4}
	$c_4=5$.
\end{prop}
\begin{proof}
	Let $A$ be an extremal matrix in $\Avoid(m,1_{4,1},F_9)$ that has a large number of 3-columns that intersect in the first two rows (which again is the only chance of a higher bound than $c_3$) and let $A'$ denote the matrix of 0, 1, and 2-columns in $A$.  If $A'$ contains an $I_2$ below the first two rows (say in rows 3 and 4 and columns $c_1$ and $c_2$ respectively), then $c_1$  and $c_2$ restricted to rows 1 and 2 must look like $\begin{bmatrix}
	1 & 0\\ 
	0 & 1
	\end{bmatrix}$ (they can't contain two 1's in these rows without being a 3-column, and if $c_1$ and $c_2$ both had 0's in one of these rows, say the first, then we could find an $F_9$ by considering rows 1, 2 and 3, columns $c_1,\ c_2$, a 3-column that has a 1 in row $i\ne  3,4$ and row $i$).  In this situation one can't have a third column $c_3$ of $A'$ with a 1 beyond the first two rows, as either $c_3$ has a 1 in row 3 (in which case it can't be equal to $\begin{bmatrix}
	1 \\ 0
	\end{bmatrix}$ in the first two rows since $c_3\ne c_1$, and hence $c_3$ and $c_1$ contain a row of 0's in the first two rows, giving an $F_9$), row 4 (symmetric argument), or some row other than 3 and 4 (in which case $c_3$ restricted to the first two rows must be $\begin{bmatrix}
	1\\ 0
	\end{bmatrix}$ to not induce an $F_9$ with $c_2$ and $\begin{bmatrix}
	0 \\ 1
	\end{bmatrix}$ to not induce an $F_9$ with $c_1$, which is impossible).  The only other columns that would be allowed are the four columns with no 1's beyond the first two rows, so in this case we have $|A'|\le 6$.
	
	The only other case to consider is when all the 1's beyond the second row lie in the same row (say the third), in which case there can be at most ${3\choose 2}+{3\choose 1}+{3\choose 0}=7$ columns of $A'$, obtained by considering all columns which have fewer than two 1's in the first three rows and no 1's outside these rows.  Such an $A'$ avoids $F_9$ (since $F_9$ requires four rows with 1's in them), so in total we have that $|A'|\le 7$ and that $|A'|=7$ can be obtained. Thus in total we have $\forb(m,1_{4,1},F_9)\le 7+(m-2)=m+5$, and this can be achieved by letting $A$ have all 0, 1 and 2-columns with fewer than three 1's in the first three rows and all 3-columns that have 1's in rows 1 and 2.
\end{proof}

\begin{cor}\label{1F9}
For sufficiently large $m$:
    \begin{align*}
        \forb(m,1_{3,1},F_9)&=m+2\\ 
        \forb(m,1_{2,2},F_9)&=m+3\\ 
        \forb(m,1_{4,1},F_9)&=m+5.
    \end{align*}
\end{cor}

\section{Avoiding $Q_9$}
\[
	Q_9=\begin{bmatrix}
	1 & 0\\ 
	1 & 0\\ 
	0 & 1\\ 
	0 & 1
	\end{bmatrix}
\]
It turns out that the problem of avoiding $Q_9$  and $1_{k,\ell}$ has a very similar flavor to the problem of avoiding $F_9$ and $1_{k,\ell}$, and because of this we will once again be able to achieve exact results.  We maintain all of our notation and terminology from the previous section.\par
The bound $\forb(m,Q_9)=\binom{m}{2}+2m-1$ was proven in \cite{ABS11}, where the following classification of $Q_9$ avoiding matrices was established (following \cite{A90}).  For each $2\le t\le m-2$  we can divide the rows
into three disjoint sets  
$A_t,B_t,C_t\subseteq\{1,2,\ldots ,m\}$ so that after permuting the 
rows the $t$-columns can either be given as
\[\mbox{type 1: }
\begin{array}{l} A_t\{\\ B_t\{\\ C_t\{\\ \end{array}
\hskip-3pt
\left[\begin{array}{l}
I_{|A_t|}\\ 1_{|B_t|, |A_t|}\\ 0_{|C_t|, |A_t|}\\
\end{array}\right]
\mbox{ or }
\mbox{type 2: }
\begin{array}{l} A_t\{\\ B_t\{\\ C_t\{\\ \end{array}
\left[\begin{array}{l}
I_{|A_t|}^c\\ 1_{|B_t|,|A_t|}\\ 0_{|C_t|,  |A_t|}\\
\end{array}\right]
.\]
We will say $t$ {\it is of type } $i$ ($i=1$ or $i=2$) if the $t$-columns are of type $i$.
\begin{lem}
	Let $m\ge 2k$, then $\forb(m,Q_9)^{=t}=m-(t-1)$ for $1<t\le k$. 
\end{lem}
\begin{proof}
	The size of a type 1 matrix of column sum t is at most $m-(t-1)$, while the size of a type 2 matrix of the same column sum is bounded by $t+1$. 
	\end{proof}
\begin{prop}\label{prop:smallt}
	Let $m\ge 2k$, then $\forb(m,Q_9,1_{k,1})=1+(k-1)m-{k-1\choose 2}$.
\end{prop}
\begin{proof}
	By the previous lemma, $\forb(m,Q_9,1_{k,1})$  is upper bounded by $1+m+\sum_{t=2}^k (m-(t-1))=1+(k-1)m-{k-1\choose 2}$, and this value can be achieved by having $m-(t-1)$  $t$-columns intersecting in the first $t-1$ rows, along with all columns of column sum 0 and 1.
\end{proof}

\begin{cor}\label{Q9}
For $m\ge 8$,
    \begin{align*}
        %\forb(m,Q_9,1_{3,1})&=2m\\ 
        \forb(m,Q_9,1_{4,1})&=3m-2.
    \end{align*}
\end{cor}

We can extend these results for $\ell>1$.

\begin{prop}
	$\forb(m,Q_9,1_{k,2})=\forb(m,Q_9,1_{k+1,1})+1$.
\end{prop}
\begin{proof}
	For the lower bound take the lower bound construction for $\forb(m,Q_9,1_{k+1,1})$ given above and add in the $(m-1)$-column with a 0 in the first row.  This new column can't be used to make a $Q_9$ since it has too few 0's, and it doesn't intersect any other column in $k$ rows so it can't be used to find a $1_{k,2}$.  Thus this new matrix is in $\Avoid(m,Q_9,1_{k,\ell})$.  For the upper bound, note that if $c,d$ are columns with at least $k+1$ 1's then either $|c\cap d|\ge k$ (in which case we have $1_{k,2}$) or there exists two rows where $c$ has 1's and $d$ does not and vice versa (in which case we have $Q_9$), so a matrix in $\Avoid(m,Q_9,1_{k,2})$ can have at most one column that has more than $k$ 1's.
\end{proof}

Analyzing the $\ell>2$ case once again turns out to be significantly more difficult than the $\ell\le 2$ cases, but nonetheless we are able to achieve some nearly tight bounds for this problem.

\begin{lem}\label{lem:medt}
	$\forb(m,Q_9,1_{k,\ell})^{=t}\le k+\ell$ for $k+\ell>t> k$.
\end{lem}
\begin{proof}
	The size of a type 1 matrix of column sum $t$ can be at most $\ell-1$ without inducing a $1_{k,\ell}$, and the size of a type 2 matrix of the same column sum is bounded by $t+1\le k+\ell$.
\end{proof}

\begin{lem}\label{L-larger}
	$\forb(m,Q_9,1_{k,\ell})^{\ge k+\ell}=\ell-1$.
\end{lem}
\begin{proof}
	Let $c$ be a column of $A\in\Avoid(m,Q_9,1_{k,\ell})^{\ge k+\ell}$ with the fewest number of 1's (say $t$ of them).  We must have $|c\cap d|\ge t-1$ for any other $d$ (as if $d$ has two 0's in rows where $c$ has 1's, by virtue of $c$ having the fewest number of 1's $d$ must have at least two 1's where $c$ has 0's, giving a $Q_9$), and hence for any other $\ell-1$ columns in $A$ there exists $k$ rows such that $c$ and all of these other columns have 1's in these rows (since each can have at most one 0 in the at least $k+\ell$ rows where $c$ has 1's), so we must have $|A|\le \ell-1$.
\end{proof}

\begin{prop}
	For $k\ge 2,\ \ell\ge 3$ and $m>(\ell+1)(k+\ell)+k$,
	\begin{align*}
		\forb(m,Q_9,1_{k,\ell})&\ge \forb(m,Q_9,1_{k+1,1})+2\ell-5\\ 
	\forb(m,Q_9,1_{k,\ell})&\le \forb(m,Q_9,1_{k+1,1})+3\ell-5.
	\end{align*}
\end{prop}
\begin{proof}
	Take the lower bound construction for  $\forb(m,Q_9,1_{k+1,1})$ and adjoin to this $\ell-2$ columns with column sum $(k+1)$ such that  $k$ of their 1's are in the first $k$ rows and their remaining 1's are in rows $k+1$ through $k+\ell-2$.  Additionally adjoin $\ell-3$ columns with column sum $(k+\ell-2)$ with $k+\ell-3$ of their 1's in the first $k+\ell-2$ rows excluding row $k$ and their remaining 1's anywhere below these rows.  One can't use a $(k+\ell-2)$-column to find a $Q_9$ (only the $(k+1)$-columns and $t$-columns with a 1 in row $k+1$ have 1's in a row where a $(k+\ell-2)$-column has a 0 in the first $(k+\ell-2)$ rows, but no such row exists beyond that for these columns, and for all other $t$-columns there exists at most one such row beyond the first $(k+\ell-2)$ and none before this) and one can't use a $(k+1)$-column either (it can't be used with a $t$-column for $t\le k+1$ as below the first $t-1$ rows of the $t$-column there aren't enough 1's), so this avoids $Q_9$.  To find a $1_{k,\ell}$, first note that at most one $t$-column with $t\le k$ could be used (as there exists no $k$ rows where two such $t$-columns both have 1's).  If one uses more than one $(k+1)$-column to find a $1_{k,\ell}$, then one must use the first $k$ rows (since these are the only rows that two distinct $(k+1)$-columns agree); but there are only $\ell-2$ $(k+1)$-columns and one $k$-column with 1's in the first $k$ rows, and no $(k+\ell-2)$-column can be used as they each have a 0 in row $k$, so one can't find $\ell$ such columns.  Thus in total one could use at most one $t$-column with $t\le k$, one $(k+1)$-column and all $\ell-3$ $(k+\ell-2)$-columns, but this can't be used to find a $1_{k,\ell}$ since there are at most $\ell-1$ columns.
	
	For the upper bound, take $A\in \Avoid(m,Q_9,1_{k,\ell})$ with $|A|\ge 1+km-{k\choose 2}$.  Let  $p$ denote the number of $k$-columns that $A$ has.   Because $\forb(m,Q_9,1_{k,\ell})^{\ge k+1}\le \ell(k+\ell)+(\ell-1)$, the only way we can have $|A|\ge 1+km-{k\choose 2}$ is if $p\ge m-k-\ell(k+\ell)-(\ell-1)$ by Proposition~\ref{prop:smallt} and Lemmas~\ref{lem:medt} and \ref{L-larger}.
	Now using that $m>(\ell+1)(k+\ell)+k$, this can only happen if columns of sum $k$ are of type 1. We assume that their common 1's are in the first $k-1$ rows, which induces an $I_p$ in
	the rows below the first $k-1$ rows.
	
	No column with at least $k+1$ 1's can have two 0's in the first $k-1$ rows (as any $k$-column has two rows where it has 0's and this large column does not, and this large column necessarily has two rows where it has 1's and the $k$-column does not, since it has at least $k+1$ 1's and two of them aren't in the first $k-1$ rows).  If a column with at least $k+1$ 1's has one 0 in the first $k-1$  rows and $k\ge 2$ then this column must cover the entire $I_p$ (otherwise we could find a column that isn't covered by the large column, take these two columns, the rows where the $k$-column has 1's and the large column has 0's and any rows that the large column has that other doesn't to find a $Q_9$), but because $I_p$ is large we can have at most $\ell-1$ columns that cover it before inducing a $1_{k,\ell}$.  We ignore these covering columns for now and restrict our attention to columns with at least $k+1$ 1's and that are identically 1 in the first $k-1$ rows. Let $c$ be such a column with the fewest number of 1's and assume it has 1's in the first $k+1$ rows.  As argued in the second lemma, any other column must have $|c\cap d|\ge k$ and in particular (since all the columns we're considering are identically 1 in the first $k-1$ rows) the only 0's the other columns can have are in the $k$th and $k+1$st rows.  There can be at most $\ell-1$ columns with a 0 in the $k$th row before inducing a $1_{k,\ell}$, but if there are precisely $\ell-1$ such columns then $A$ can not contain the $k$-column with 1's in rows 1 through $k-1$ and row $k+1$, decreasing the maximum value $p$ can take by 1, so ``effectively'' these columns can contribute at most $\ell-2$.  Similar results hold for columns with a 0 in the $k+1$st row, so in total we have $|A|\le \forb(m,Q_9,1_{k+1,1})+2(\ell-2)+\ell-1=\forb(m,Q_9,1_{k+1,1})+3\ell-5$
\end{proof}

We can get a slightly larger lower bound when $k$ is sufficiently large.
\begin{prop}
    If $\ell=3$ and $k\ge 3$ or if $k\ge \ell-1\ge 3$ then \[\forb(m,Q_9,1_{k,\ell})\ge \forb(m,Q_9,1_{k+1,1})+2\ell-3.\]
\end{prop}
\begin{proof}
    If $k\ge \ell-1$ then take the lower bound construction for $\forb(m,Q_9,1_{k+1,1})$ and adjoin to this $\ell-2$ columns with column sum $(k+1)$ with  $k$ of their 1's in the first $k$ rows and also adjoin $\ell-1$ $(m-1)$-columns with their 0's in the first $\ell-1$ rows (which by assumption is in the first $k$ rows). None of the $(m-1)$-columns can be used to find a $Q_9$ (as they have too few 0's), and by the same logic as before neither can the $(k+1)$-columns.  To find a $1_{k,\ell}$, again note that at most one $t$-column with $t\le k$ could be used and if one uses more than one $(k+1)$-column to find a $1_{k,\ell}$, then one must use the first $k$ rows which means no $(m-1)$-column can be used (since each has a 0 in the first $k$ rows), so again we conclude that at most one $(k+1)$-column can be used.  One can't use only $(m-1)$-columns since there are at most $\ell-1$ of them, but if any two $(m-1)$-columns are used then one can't use two of the first $k$ rows (since each has a different 0 in these rows), and hence one can't use any of the $t$-columns with $t\le k+1$ (since outside of these rows they have at most $k-1$ 1's).  Thus the only way one can find a $1_{k,\ell}$ is to use one $(m-1)$-column, one $(k+1)$-column and one $k$-column.  If $\ell\ge 4$ then we clearly can not find a $1_{k,\ell}$, but if $\ell=3$ and $k=2$ one could use the $2$-column with 1's in row 1 and row 3, the $3$-column with 1's in rows 1 through 3, and the $(m-1)$-column with a 0 in row 2 to find a $1_{2,3}$.  If $k\ge \ell=3$ then each $(m-1)$-column and $k$-column only share $k-1$ rows with 1's in both columns, so in this case we avoid $1_{k,\ell}$.
\end{proof}

\section{Future Directions}

A natural extension to this work would be to consider all simple minimal cubic configurations, not just those with 4 rows.  \cite{survey} does not explicitly list these configurations, but it is possible to determine the complete list (provided a certain conjecture is true).

First, note that there exists no minimal cubic configuration with 7 or more rows.  Indeed, each column of a 7 rowed matrix contains $1_{4,1}$ or $0_{4,1}$, meaning the configuration can't be a minimal cubic.

\begin{conj}\label{5conj}
    There exists no 5-rowed minimal cubic configuration.
\end{conj}
\begin{prop}
    Conjecture~\ref{5conj} holds provided Conjecture~8.1 of  \cite{survey} is true.
\end{prop}
\begin{proof}
   Indeed, if Conjecture~8.1 holds then we need only consider the configurations $F'_{12},\ldots,F'_{24}$ (where $F'_i$ in our notation corresponds to $F_i$ of \cite{survey}).  We note that $1_{4,1}\prec F'_{12},\ 0_{4,1}\prec F'_{13},\ F_{9}^c\prec F'_{14},F'_{22},\ F_{9}\prec F'_{15},F'_{23},\ F_{10}^c\prec F'_{16},\ F_{10}\prec F'_{17},\ F_{11}\prec F'_{21},F'_{24}$, and thus none of these configurations can be minimal.  
\end{proof}
\begin{prop}
   The configurations $F_{14}$ and $F_{15}$ listed below are minimal cubic configurations.  Moreover, they are the only simple 6-rowed minimal cubic configurations.
\end{prop}

\begin{center}
\captionof{table}{Minimal Simple Cubic Configurations with 6 Rows}\label{tab:mincub6}
\begin{tabular}{ | c | c | c |  c| c|}
\hline 
 & Configuration  $F_i$ & Quadratic Const.(s) & Cubic Const.(s) & Proposition\\ 
\hline
$F_{14}$ & $\begin{bmatrix} 1 &  0 \\  1 &  0 \\ 1 &  0 \\ 0 &  1 \\ 0 &  1 \\ 0 &  1
	\end{bmatrix}$ & $\begin{matrix}I\times I\\ I\times I^c\\ I\times T\\ I^c\times I^c\\ I^c\times T\end{matrix}$ & $\begin{matrix}I\times I\times T\\ I\times I^c\times T\\ I^c\times I^c\times T\end{matrix}$ & Prop.~\ref{P14}\\ 
\hline
$F_{15}$ & $\begin{bmatrix}
	1 &  0 & 0\\ 
	0 &  1  & 0\\ 
	0 &  0  & 1\\ 
	0 &  1 & 1 \\ 
	1 &  0 & 1\\ 
	1 &  1 & 0
	\end{bmatrix}$ & $\begin{matrix}I\times I\\ I\times T\\ I^c\times I^c\\ I^c\times T\\ T\times T\end{matrix}$ & $\begin{matrix}I\times I\times T\\ I^c\times I^c\times T\end{matrix}$ & Prop.~\ref{P15}\\ 
\hline
\end{tabular}
\end{center}

\begin{proof}
Note that we need only consider configurations whose column sum's are precisely 3, as otherwise the configuration will not be minimal.  It is noted in \cite{ARS14} that the following configurations are the only six-rowed simple matrices with at least a cubic lower bound such that removing any column would make the configuration less than cubic:

\begin{align*}
	F_{14},\ F_{15},\ F_{16}=\begin{bmatrix}
	1 &  1 & 1\\ 
	1 &  1 & 1\\ 
	1 &  0  & 0\\ 
	0 &  1 & 0\\ 
	0 &  0 & 1\\ 
	0 &  0 & 0
	\end{bmatrix},\ F_{16}^c,\ F_{17}=\begin{bmatrix}
	1 & 1 & 1\\ 
	1 & 1 & 0 \\
	1 & 0 & 0\\ 
	0 & 1 & 0\\ 
	0 & 0 & 1\\ 
	0 & 0 & 1
	\end{bmatrix},\ F_{17}^c.
\end{align*}
Note that $F_{10}\prec F_{16}$ and $F_9\prec F_{17}$, and consequently $F_{10}^c\prec F_{16}^c$ and $F_9^c\prec F_{17}^c$.  Thus the only configurations that could be minimal cubics are $F_{14}$ and $F_{15}$.
	
Anstee and Keevash in \cite{AK06} note that $F_{14}$ is cubic, and moreover, that it with any row removed is quadratic, so this is a minimal cubic configuration.  \cite{survey} notes that the following configuration is quadratic:
\[
    F_7=\begin{bmatrix}
        1 & 1 & 0 & 1 & 1 & 0\\ 
        1 & 0 & 1 & 1 & 1 &1\\ 
        0 & 1 & 0 & 1 & 0 & 1\\ 
        0 & 0 & 1 & 0 & 0 & 1\\ 
        0 & 0 & 0 & 0 & 1 & 0
    \end{bmatrix}
\]
If $F'_7$ consists of the 2nd, 3rd and 5th columns of $F_7$ then we note that $F'_7$ is $F_{15}$ without one of its rows (so if $F_{15}$ is a cubic configuration it must be a minimal cubic).  If we apply the standard induction for $\forb(m,F_{15})$,  we must have $F'_7\nprec C_r$ (as otherwise $F_{15}\prec [01]\times F_7'\prec A$), and hence $|C_r|=O(m^2)$, so we conclude that $forb(m,F_{15})=O(m^3)$.
\end{proof}

\begin{prop}\label{P14}
    $F_{14}\nprec I\times I,I\times I^c,I\times T,I^c\times I^c,I^c\times T$ and $F_{14}\nprec I\times I\times T,I\times I^c\times T,I^c\times I^c\times T$.  Moreover, these are the only 2 and 3-fold products that avoid $F_{14}$.
\end{prop}
\begin{proof}
    Note that any selection of three rows of $F_{14}$ contains $1_{2,1}$ and $0_{2,1}$, but neither $I$ nor $I^c$ contains both of these configurations so any $I$ or $I^c$ in a product could contribute at most 2 rows to find $F_{14}$.  Similarly, any four rows of $F_{14}$ contains $I_2$, and hence $T$ can contribute at most 3 rows in finding $F_{14}$ for any product it is involved in.  This shows that all 2-fold products except possibly $T\times T$ avoids $F_{14}$, but it isn't too difficult to see that $F_{14}\prec T_4\times T_4\prec T\times T$.
    
    Any 3-fold product involving only $I$'s and $I^c$'s will contain $F_{14}$, as each of these can contribute an $I_2$ from two of their rows and three of these put together give $F_{14}$.  Thus the only possible 3-fold product that could avoid $F_{14}$ are products using precisely one $T$ and the rest $I$'s and $I^c$'s.  And this does in fact avoid $F_{14}$, as the most each $I$ and $I^c$ can contribute is two rows that form an $I_2$, but this still leaves at least one $I_2$ to be covered by the $T$, which  it can not do.
\end{proof}

\begin{prop}\label{P15}
    $F_{15}\nprec I\times I,I\times T,I^c\times I^c,I^c\times T,T\times T$ and $F_{15}\nprec I\times I\times T,I^c\times I^c\times T$.   Moreover, these are the only 2 and 3-fold products that avoid $F_{15}$.
\end{prop}
\begin{proof}
    As $F_{15}$ consists of an $I_3$ on top of an $I^c_3$, it is clear that $F_{15}\prec I\times I^c$.  Note that $I_3^c\nprec I\times I,I\times T,T\times T$, and hence $F_{15}$ will not be contained in any of these products.  Similarly $I_3\nprec I^c\times I^c$ implies that $F_{15}\nprec I^c\times I^c,I^c\times T$.
    
    To see that $F_{15}\nprec I\times I\times T$, note that any two rows of the $I_3^c$ of $F_{15}$ contains $1_{2,1}$ (so $I$ can contribute to at most one row of $I_3^c$) and $I_2$ (so $T$ can contribute to at most one row of $I_3^c$).  Consequently, each of the $I$'s and the $T$ must contribute to precisely one row of the $I_3^c$.  But if an $I$ contributes to the $i$th row of $F_{15}$ ($i\ge 4$), then the only other row it can contribute to is the $(i-3)$rd row (as using any other row gives a $1_{2,1}$).  But if $T$ covers the $i$th row ($i\ge 4$), it can not also contribute to the $(i-3)$rd row, as these two rows contain an $I_2$.  Thus no matter which rows of the $I_3^c$ the $I$ and $T$ blocks cover, it will be impossible to cover all 6 rows of $F_{15}$.  It is not difficult to show that $F_{15}\prec I\times T\times T$ by finding rows 1 and 4 in $I$, rows 3 and 5 in the first $T$ and rows 2 and 6 in the second $T$.  Similarly $F_{15}\prec T\times T\times T$ by finding rows 1 and 5 in one T, 2 and 6 in another, and 3 and 4 in the last.
\end{proof}

From these constructions we are able to show that $\forb(m,Q,F)=\Theta(m^2)$ where $Q$ is a minimal quadratic configuration and $F$ is either $F_{14}$ or $F_{15}$ with the exception of the pairing $Q=Q_8$ and $F=F_{14}$ (as the only 2-fold product that avoids $Q_8$ is $T\times T$, which  is the only 2-fold product that contains $F_{14}$).  We would predict based on our previous work that $\forb(m,Q_8,F_{14})=o(m^2)$, but we are unable to show this.

\begin{quest}
    What is $\forb(m,Q_8,F_{14})$?
\end{quest}

The problem of pairing $F_{14}$ and $F_{15}$ with other cubics is also a difficult question.  Through the constructions we listed, it is possible to show that $\forb(m,F_1,F_2)=\Omega(m^2)$ for $F_1$ either $F_{14}$ and $F_{15}$ and $F_2$ any other simple minimal cubic configuration, and that $\forb(m,F_{14},F_{15})=\Theta(m^3)$, as well as $\forb(m,F_1,F_2)=\Theta(m^3)$ where $F_1$ is $F_{14}$ or $F_{15}$ and $F_2$ is $F_{12}$ or $F_{12}^c$.  Unfortunately, we are unable to prove any tighter bounds.

\begin{quest}
    What is $\forb(m,F_1,F_2)$ in general for $F_1=F_{14}$ or $F_{15}$ and $F_2$ any simple minimal cubic configuration?
\end{quest}

One potential route for proving these results, at least for $F_{14}$, would be to characterize how matrices in $A\in \Avoid(m,F_{14})^{=t}$ must look like as was done for $Q_9$ in \cite{ABS11}.  However, classifying $t$-columns of $F_{14}$ seems to be a more difficult problem compared to $Q_9$.

\begin{quest}
    Is there a nice characterization of matrices $A\in \Avoid(m,F_{14})^{=t}$?
\end{quest}

\end{document}